\documentclass[11pt]{article}
\usepackage{times}
\usepackage{amsopn,amssymb,amsthm,amsmath,bm}
\usepackage{paralist}
\usepackage{algorithm,algorithmic}
\usepackage{color}
\usepackage{graphicx}
\usepackage{comment}
\usepackage[urlcolor=black]{hyperref}
\hypersetup{
     colorlinks   = true,
     linkcolor    = blue,
     citecolor    = green
}
\usepackage{multirow}
\usepackage{float}
\usepackage{subfigure}
\usepackage{enumerate}
\usepackage[style=alphabetic,natbib=true,backend=bibtex]{biblatex}
\usepackage{tablefootnote}
\usepackage{epstopdf}

\bibliography{references}

\newlength{\defbaselineskip}
\newtheorem{theorem}{Theorem}

\newtheorem{definition}[theorem]{Definition}
\newtheorem{lemma}[theorem]{Lemma}

\newtheorem{corollary}[theorem]{Corollary}
\newtheorem{assumption}{Assumption}

\newenvironment{tableprime}[1]
  {\renewcommand{\thetable}{\ref{#1}$'$}%
   \addtocounter{table}{-1}%
   \begin{table}[H]}
  {\end{table}}

\def\C{\mathcal{C}}
\def\K{\mathcal{K}}

\newif\ifwithcomments
\withcommentstrue
\ifwithcomments
\newcommand{\xxx}[1]{{\color{red} (#1)}}
\else
\newcommand{\xxx}[1]{}
\fi

\newcommand{\defeq}{\mathrel{\mathop:}=}

\DeclareMathOperator{\bigO}{\mathcal{O}}

\DeclareMathOperator{\nnz}{nnz}

\DeclareMathOperator{\argmin}{argmin}

\newcommand{\Expect}[1]{\mbox{}{\mathbf{E}}\left[#1\right]}

\def\reals{\mathbb{R}}

\def\lmax{\lambda_{\max}}
\def\lmin{\lambda_{\min}}

\def\tr{\textbf{tr}}
\newcommand{\vv}[1] {\mathbf{#1}}

\def\g{\vv{g}}
\def\u{\vv{u}}
\def\v{\vv{v}}

\def\x{\vv{x}}
\def\y{\vv{y}}
\def\w{\vv{w}}

\def\A{\vv{A}}

\def\a{\vv{a}}
\def\Q{\vv{Q}}
\def\X{\vv{X}}
\def\Y{\vv{Y}}
\def\U{\vv{U}}
\def\R{\vv{R}}

\def\D{\vv{D}}
\def\S{\vv{S}}
\def\M{\vv{M}}
\def\H{\vv{H}}
\def\L{\vv{L}}

\newcommand\bbR{\ensuremath{\mathbb{R}}} 

\begin{document}

\title{\Large Sub-sampled Newton Methods with Non-uniform Sampling}
\author{
  Peng Xu
  \thanks{
    Institute for Computational and Mathematical Engineering,
    Stanford University,
    Email: \{pengxu, jiyan\}@stanford.edu
  }
  \and
  Jiyan Yang
  \footnotemark[1]
  \and
  Farbod Roosta-Khorasani
  \thanks{
    International Computer Science Institute and Department of Statistics, University of California at Berkeley,    
    Email: \{farbod,mmahoney\}@stat.berkeley.edu
  }
  \and
 Christopher R{\'e}
 \thanks{
    Department of Computer Science,
    Stanford University,
    Email: chrismre@cs.stanford.edu
  }
  \and
  Michael W. Mahoney
  \footnotemark[2]
}

\date{\today}
\maketitle

\begin{abstract}
We consider the problem of finding the minimizer of a convex function $F: \reals^d \rightarrow \reals$ of the form $F(\w) \defeq \sum_{i=1}^n f_i(\w) + R(\w)$ where a low-rank factorization of $\nabla^2 f_i(\w)$ is readily available.
We consider the regime where $n \gg d$. 
As second-order methods prove to be effective in finding the minimizer to a high-precision, in this work, we propose randomized Newton-type algorithms that exploit \textit{non-uniform} sub-sampling of $\{\nabla^2 f_i(\w)\}_{i=1}^{n}$, as well as inexact updates, as means to reduce the computational complexity.
Two non-uniform sampling distributions based on {\it block norm squares} and {\it block partial leverage scores} are considered in order to capture important terms among $\{\nabla^2 f_i(\w)\}_{i=1}^{n}$. 
We show that at each iteration non-uniformly sampling at most $\bigO(d \log d)$ terms from $\{\nabla^2 f_i(\w)\}_{i=1}^{n}$ is sufficient to achieve a linear-quadratic convergence rate in $\w$ when a suitable initial point is provided.
In addition, we show that our algorithms achieve a lower computational complexity and exhibit more robustness and better dependence on problem specific quantities, such as the condition number, compared to similar existing methods, especially the ones based on uniform sampling.
Finally, 
we empirically demonstrate that our methods are at least twice as fast as Newton's methods with ridge logistic regression on several real datasets.
\end{abstract}

\section{Introduction}
\label{sec:intro}
Consider the following optimization problem
\begin{equation}
 \label{eq:p}
    \min_{\w \in \C} F(\w) \defeq \sum_{i=1}^n f_i(\w) + R(\w)
\end{equation}
where $f_i(\w), R(\w)$ are smooth functions and $\C \subseteq \reals^d$ is a convex constraint set. Many machine learning and scientific computing problems involve finding an approximation of the minimizer of the above optimization problem to \emph{high precision}. For example, consider any machine learning application where each $f_{i}$ is a loss function corresponding to $i^{th}$ data point and $F$ is the \emph{empirical risk}. The goal of solving~\eqref{eq:p} is to obtain a solution with small generalization error, i.e., high predictive accuracy on ``unseen'' data. One can show that the generalization error of the empirical risk minimizer is to within $\mathcal{O}(1/\sqrt{n})$ additive error of that of the true population risk minimizer; e.g., see~\cite{bousquet2002stability,cesa2004generalization,vapnik1989inductive,vapnik2013nature}. As a result, in large scale regime that we consider in this paper where there are many data points available, i.e., $n\gg 1$, obtaining a solution to~\eqref{eq:p} to high precision would indeed guarantee a low generalization error.
A second example is that in many problems in which the optimization variable $\w$ contains specific meanings, a high-precision solution to~\eqref{eq:p} is necessary for interpreting the results.
Examples of such settings arise frequently in machine learning such as sparse least squares~\cite{tibshirani1996regression}, generalized linear models (GLMs)~\cite{friedman2001elements} and metric learning problems~\cite{kulis2012metric} among many more modern large scale problems.

There is a plethora of first-order optimization algorithms~\cite{bubeck2014theory,nocedal2006numerical} for solving~\eqref{eq:p}. 
However, 
for ill-conditioned problems, it is often the case that first-order methods return a solution far from the minimizer, $\w^\ast$, albeit a low objective value. (See Figure~\ref{fig:app} in Section~\ref{sec: experiments} for example.)
On the other hand, most second-order algorithms prove to be more robust to such 
ill conditioning.
This is so since, using the curvature information, second-order methods properly rescale the gradient, such that it is a more appropriate direction to follow. 
For example, take the canonical second-order method, i.e., {\it Newton's method}, which, in the unconstrained case, has updates of the form 
\begin{eqnarray}
\label{eq:newton}\w_{t+1} = \w_{t} - [ \H(\w_t) ]^{-1} \g(\w_t),
\end{eqnarray}
where $\g(\w_t)$ and $\H(\w_t)$ denote the gradient and the Hessian of $F$ at $\w_{t}$, respectively.
Classical results indicate that under certain assumptions, Newton's method can achieve a locally super-linear convergence rate, which can be shown to be \textit{problem independent}! Nevertheless, the cost of forming and inverting the Hessian is a major drawback in using Newton's method in practice. 

In this regard, there has been a long line of work that tries to provide sufficient second-order information with feasible computations. 
For example, among the class of quasi-Newton methods, the BFGS algorithm~\citep{nocedal2006numerical} and its limited memory version~\citep{liu1989lbfgs} are the most celebrated.
However, the convergence guarantee of these methods can be much weaker than Newton's methods.
More recently, authors in~\citep{pilanci2015newton,erdogdu2015convergence,roosta2016sub1} considered using sketching and sampling techniques to construct an approximate Hessian matrix and using it in the update rule~\eqref{eq:newton}. They showed that such algorithms inherit a local linear-quadratic convergence rate with a substantial computational gain.


In this work, we propose novel, robust and highly efficient non-uniformly sub-sampled Newton methods (SSN) for a large sub-class of problem~\eqref{eq:p}, where the Hessian of $F(\w)$ in~\eqref{eq:p} can be written as 
\begin{equation}
 \H(\w) = \sum_{i=1}^n \A^T_i(\w) \A_i(\w) + \Q(\w),
\end{equation}
where $\A_i(\w) \in \reals^{k_i \times d}, \; i = 1,2,\ldots,n,$ are readily available and $\Q(\w)$ is some positive semi-definite matrix.
This situation arises very frequently in many applications such as machine learning. For example, take any problem where $f_i(\w) = \ell(\x_i^T \w)$, $\ell(\cdot)$ is any convex loss function and $\x_i$'s are data points, and $R(\w) = \frac{\lambda}{2} \|\w\|^2$. In such situations, $\A_i(\w)$ is simply $\sqrt{\ell''(\x_i^T \w)} \x_i^{T}$ and $\Q(\w) = \lambda \mathbf{I}$.

First, we choose a sampling scheme $\mathcal{S}$ that constructs an appropriate non-uniform sampling distribution $\{p_i\}_{i=1}^n$ over $\{\A_{i}(\w)\}_{i=1}^{n}$ 
and samples $s$ terms from $\{\A_{i}(\w)\}_{i=1}^{n}$.
The approximate Hessian, constructed as $\tilde \H(\w_t) = \sum_{i \in \mathcal{I}} \A_i^T(\w_t) \A_i(\w_t)/p_i + \Q(\w_t)$ where $\mathcal{I}$ denotes the set of sub-sampled indices, is then used to update the current iterate as 
\begin{eqnarray}
\label{eq:sub0}
\w_{t+1} = \w_{t} - [\widetilde \H(\w_t)]^{-1} \g(\w_t).
\end{eqnarray}
Second, when the dimension of the problem, i.e., $d$, is so large that solving the above linear system, i.e., \eqref{eq:sub0}, becomes infeasible, we consider solving \eqref{eq:sub0} inexactly by using an iterative solver $\mathcal{A}$, e.g., Conjugate Gradient or Stochastic Gradient Descent, with a few iterations such that a high-quality approximate solution can be produced with a less complexity. Such inexact updates used in many second-order optimization algorithms have been well studied in \cite{byrd2011use, dembo1982inexact}.



Under certain conditions, it can be shown that this type of randomized Newton-type algorithm can achieve a local linear-quadratic convergence rate shown as below (formally stated in Lemma~\ref{lem:general})
\begin{equation}
\label{eq: rate}
\|\w_{t+1} - \w^*\|_2 \le C_q \|\w_t - \w^*\|_2^2 + C_l \|\w_t - \w^*\|_2,
\end{equation}
where $C_q, C_l$ are some constants that can be controlled by the Hessian approximation quality, i.e., choice of sampling scheme $\mathcal{S}$,\footnote{To be more precise, by a sampling scheme here, we mean the way we construct the sampling distribution $\{p_i\}_{i=1}^n$, e.g., uniform sampling distribution or leverage scores sampling distribution, and the value of sampling size $s$.} and the solution quality of~\eqref{eq:sub0}, i.e., choice of solver $\mathcal{A}$.\footnote{By a solver here, we mean the choice of the specific algorithm with parameters specified, e.g., number of iterations, we use to obtain a high-quality approximation solution to the subproblem.}

Different choices of sampling scheme $\mathcal{S}$ and solver $\mathcal{A}$ lead to different complexities in SSN. Below, we briefly discuss their effects.
As can be seen, the total complexity of our algorithm can be characterized by the following three factors, each of which is affected by $\mathcal{S}$ or $\mathcal{A}$, or both.
\begin{compactitem}
\item Number of total iterations $T$ determined by the convergence rate, i.e., $C_q$ and $C_l$ in \eqref{eq: rate} which is affected by sampling scheme $\mathcal{S}$ and solver $\mathcal{A}$.
\item In each iteration, the time $t_{const}$ it needs to construct $\{p_i\}_{i=1}^n$ and sample $s$ terms, which is determined by sampling scheme $\mathcal{S}$.
\item In each iteration, the time $t_{solve}$ it needs to (implicitly) form $\tilde \H$ which is affected by sampling scheme $\mathcal{S}$ 
and to (inexactly) solve the linear problem~\eqref{eq:sub0} which is affected by solver $\mathcal{A}$.
\end{compactitem}
With these, the total complexity can be expressed as
\begin{equation}
 \label{eq:complexity0}
  T \cdot (t_{const} + t_{grad} + t_{solve}),
\end{equation}
where $t_{grad}$ is the time it takes to compute the full gradient $\nabla F(\w_t)$ which is not affected by the choice of $\mathcal{S}$ and $\mathcal{A}$ and will not be discussed in the rest of this paper. 


As discussed above, the choice of sampling scheme $\mathcal{S}$ and solver $\mathcal{A}$ plays an important role in our algorithm. Below, we focus on $\mathcal{S}$ first and discuss a few concrete sampling schemes.
A natural and simple approach is uniform sampling discussed in \citep{erdogdu2015convergence,roosta2016sub1, roosta2016sub2}.
The greatest advantage of uniform sampling is its simplicity of construction. 
However, in the presence of high non-uniformity among $\{\nabla^{2} f_i(\w)\}_{i=1}^{n}$, the sampling size required to sufficiently capture the curvature information of the Hessian can be very large, which makes the resulting algorithm less beneficial.

In this paper, we consider two non-uniform sampling schemes, \textit{block norm squares} and a new, and more general, notion of leverage scores named {\it block partial leverage scores} (Definition~\ref{def:bplev}); see Section~\ref{sec: main} for detailed construction.
The motivation for using these schemes is that we can in fact view the sufficient conditions for achieving the local linear-quadratic convergence rate~\eqref{eq: rate} as matrix approximation guarantees and these two sampling schemes can yield high-quality matrix approximations; see Section~\ref{subsec: conditions} for more details.
Recall that, the choice of $\mathcal{S}$ affects the three terms, namely, $T$ (manifested in $C_q$ and $C_l$), $t_{const}$, $t_{solve}$, in \eqref{eq:complexity0}. By leveraging and extending theories in randomized linear algebra, we can show how these terms may become when different sampling schemes are used as presented in Table~\ref{tab: results}. Detailed theory is elaborated in Section~\ref{sec: theory}.
As we can see, block norm squares sampling and leverage scores sampling require a smaller sampling size $s$ than uniform sampling does since $t_{solve} = sd^2$. Furthermore, the dependence of $C_q$ and $C_l$ on the condition number $\kappa$ reveals that leverage scores sampling is more robust to ill-conditioned problems; see Section~\ref{sec:unif_comp} for more detailed discussions.
 
\begin{table}
\centering
\begin{tabular}{c|c|c|cc|c}
\sc 
name & $t_{const}$ & $t_{solve} = sd^2$ & $C_q$ & $C_l$ & reference  \\
\hline
Newton's method & $0$ & $\bigO(nd^2)$ & $\tilde \kappa$ & $0$ & \cite{boyd2004convex} \\
SSN (leverage scores)& $\bigO(\nnz(\A) \log n)$ & $\tilde \bigO(d^3/\epsilon^2)$ & $\frac{\tilde \kappa}{1-\epsilon}$ & $\frac{\epsilon\sqrt{\kappa}}{1-\epsilon} $ &\textbf{This paper}  \\
SSN (block norm squares)& $\bigO(\nnz(\A))$ & $\tilde \bigO({\bf sr}(\A)d^2/\epsilon^2)$ & $\frac{\tilde \kappa}{1-\epsilon\kappa} $ & $\frac{\epsilon\kappa}{1-\epsilon\kappa}$ &\textbf{This paper} \\
SSN (uniform)& $\bigO(1)$ &  $\tilde\bigO\left(nd^2\frac{\max_i \|\A_i\|^2}{ \|\A\|^2}/\epsilon^2\right)$ & $\frac{\tilde \kappa}{1-\epsilon\kappa} $ & $\frac{\epsilon\kappa}{1-\epsilon\kappa}$ & \cite{roosta2016sub2}
\end{tabular}
\caption{Comparison between Newton's method and sub-sampled Newton methods (SSN) with different sampling schemes. In the above, $C_q$ and $C_l$ are the constants achieved in~\eqref{eq: rate}; $\tilde \kappa$ and $\kappa$ are constants related to the problem only defined in~\eqref{eq:constants}; $\A \in \reals^{\bigO(n)\times d}$ is a matrix that satisfies $\A_i^T \A_i = \H_i(\w)$ and $\A^T \A = \sum_{i=1}^n \H_i(\w)$; ${\bf sr}(\A)$ is the stable rank of $\A$ satisfying ${\bf sr}(\A) \leq d$; $\nnz(\A)$ denote the number of non-zero elements in $\A$.
Note here, to remove the effect of solver $\mathcal{A}$, we assume the subproblem~\eqref{eq:sub0} is solved exactly. Also, we assume the problem is unconstrained ($\C = \reals^d$) so that $t_{solve} = sd^2$.
}
\label{tab: results}
\end{table}

Next, we discuss the effect of the solver $\mathcal{A}$. Typically, a direct solver takes $\bigO(sd^2)$ time to solve the subproblem~\eqref{eq:sub0} where $s$ is the sampling size. This becomes prohibitive when $d$ is large. However, iterative solvers allow one to obtain a high quality approximation solution with a few iterations which may drastically drive down the complexity. For example, Conjugate Gradient (CG) takes $\bigO(sd\sqrt{\kappa} \log(1/\epsilon_0))$ to return an approximate solution with relative error $\epsilon_0$ where $\kappa$ is the condition number of the problem. In Lemma~\ref{lem:general} we show that such inexactness will not deteriorate the performance of SSN too much.

Indeed, based on \eqref{eq: rate}, it is possible to choose the parameters, e.g., sampling size $s$ in the sampling scheme $\mathcal{S}$ and number of iterations to run in solver $\mathcal{A}$, so that SSN converges in a constant linear rate, e.g.,
\begin{equation}
\label{eq:linear_rate0}
  \|\w_{t+1} - \w^\ast \| \leq \frac{1}{2} \|\w_t - \w^\ast \|. 
\end{equation}
By this way, the complexity per iteration of SSN can be explicitly given. In Table~\ref{tab: complexity} we summarize these results with comparison to other stochastic second-order approaches such as LiSSA~\citep{agarwal2016second}.
It can be seen from Table~\ref{tab: complexity} that compared to Newton's methods, these stochastic second-order methods trade the coefficient of the leading term $\bigO(nd)$ with some lower order terms that only depend on $d$ and condition numbers. 
Although SSN with non-uniform sampling has a quadratic dependence on $d$, its dependence on the condition number is better than the other methods. There are two main reasons.
First, the total power of the condition number is lower, regardless of the versions of the condition number needed.
Second, SSN (leverage scores) and SSN (block norm squares) only depend on $\kappa$ which can be significantly lower than the other two definitions of condition number, i.e., $\hat \kappa$ and $\bar \kappa$; see Section~\ref{sec:complexity_comp} for more details.
%

\begin{table}
\centering

\begin{tabular}{c|c|c}
\sc 
name & \sc complexity per iteration & \sc reference \\
\hline
Newton-CG method & $\tilde \bigO(\nnz(\A)\sqrt{\kappa})$ & \cite{nocedal2006numerical} \\
SSN (leverage scores) & $\tilde \bigO(\nnz(\A)\log n + d^2\kappa^{3/2})$ & {\bf This paper} \\
SSN (block norm squares) &  $\tilde \bigO(\nnz(\A) + {\bf sr}(\A)d\kappa^{5/2} )$ & {\bf This paper} \\
Newton Sketch (SRHT) & $\tilde\bigO(nd (\log n)^4 + d^2(\log n)^4 \kappa^{3/2})$ & \cite{pilanci2015newton} 
\\
SSN (uniform) & $\tilde\bigO(\nnz(\A) + d\hat\kappa \kappa^{3/2})$ & \cite{roosta2016sub2} \\
LiSSA & $\tilde\bigO(\nnz(\A) + d \hat\kappa \bar \kappa^2)$ & \cite{agarwal2016second}
\end{tabular}
\caption{
Complexity per iteration of different methods to obtain a problem independent local linear convergence rate. The quantities $\kappa$, $\hat \kappa$, and $\bar \kappa$ are the local condition numbers, defined in~\eqref{eq:cond_nums} at the optimum $\w^\ast$, satisfying 
$\kappa \leq \hat\kappa \leq \bar\kappa$; $\A \in \reals^{\bigO(n)\times d}$ is a matrix that satisfies $\A_i^T \A_i = \H_i(\w)$ and $\A^T \A = \sum_{i=1}^n \H_i(\w)$; ${\bf sr}(\A)$ is the stable rank of $\A$ satisfying ${\bf sr}(\A) \leq d$; $\nnz(\A)$ denote the number of non-zero elements in $\A$. 
Note here, for the ease of comparison, we assume $\C = \reals^d$, $R(\w) = 0$, and CG is used for solving sub-problems in SSN so that the complexities can be easily expressed.}
\label{tab: complexity}
\end{table}


As we shall see (in Section~\ref{sec: experiments}), our algorithms converge much faster than other competing methods with ridge logistic regression. In particular, on several real datasets with a moderately high condition number and large $n$, our methods are at least twice as fast as Newton's methods in finding a medium- or high-precision solution, while other methods including first-order methods converge slowly. Indeed, this phenomenon is well supported by our theoretical findings---the complexity of our algorithms has a lower dependence on the problem condition number and is immune to any non-uniformity among $\{\A_{i}(\w)\}_{i=1}^{n}$ which may cause a factor of $n$ in the complexity (Table~\ref{tab: complexity}). 
In the following we present other prior work and details of our main contributions.

\subsection{Related work}
Recently, within the context of randomized second-order methods, many algorithms have been proposed that aim at reducing the computational costs involving pure Newton's method. Among them, algorithms that employ uniform sub-sampling constitute a popular line of work~\cite{byrd2011use, erdogdu2015convergence, martens2010deep, vinyals2011krylov}. In particular~\citep{roosta2016sub1,roosta2016sub2} consider a more general class of problems and, under a variety of conditions, thoroughly study the local and global convergence properties of sub-sampled Newton methods where the gradient and/or the Hessian are uniformly sub-sampled.
Our work here, however, is more closely related to a recent work~\citep{pilanci2015newton} (Newton Sketch), which considers a similar class of problems and proposes sketching the Hessian using random sub-Gaussian matrices or randomized orthonormal systems.
Furthermore, \cite{agarwal2016second} proposes a stochastic algorithm (LiSSA) that, for solving the sub-problems, employs some unbiased estimators of the inverse of the Hessian.

The main technique used by \citep{pilanci2015newton,roosta2016sub1,roosta2016sub2} and our work is sketching, which is a powerful technique in randomized linear algebra and many other applications~\citep{woodruff2014sketching,mahoney2011randomized,yang2016review}. As mentioned above, Hessian approximation can be viewed as a matrix approximation problem.
In terms of this, error analysis of matrix approximation based on leverage scores sampling has been well studied \citep{drineas2008relative, drineas2012fast, cohen2015ridge}. \citet{holodnak2015randomized} show the lower bounds of sampling size for both block norm squares sampling and leverage scores sampling in approximating the Gram product matrix. Also \citep{cohen2015uniform} implies that uniform sampling cannot guarantee spectral approximation when the sampling size is only dependent on the lower dimension.

\subsection{Main contributions}
The contributions of this paper can be summarized as follows.
\begin{compactitem}

\item For the class of problems considered here, unlike the uniform sampling used in~\cite{byrd2011use,erdogdu2015convergence,roosta2016sub1,roosta2016sub2}, we employ two non-uniform sampling schemes based on \textit{block norm squares} and a new, and more general, notion of leverage scores named {\it block partial leverage scores} (Definition~\ref{def:bplev}). It can be shown that in the case of extreme non-uniformity among $\{\A_{i}(\w)\}_{i=1}^{n}$, uniform sampling might require $\Omega(n)$ samples to capture the Hessian information appropriately. However, we show that our non-uniform sampling schemes result in sample sizes completely \textit{independent of $n$} and are immune to such non-uniformity. 

\item 
Within the context of sub-sampled Newton-type algorithms,~\cite{byrd2011use,roosta2016sub2} incorporate inexact updates where the sub-problems are solved only approximately and study global convergence properties of their algorithms. We extend the study of inexactness to the finer level of local convergence analysis.
\item 
We provide a general structural result (Lemma~\ref{lem:general}) showing that, as in~\cite{erdogdu2015convergence, pilanci2015newton, roosta2016sub1}, our main algorithm exhibits a linear-quadratic {\it solution} error recursion. However, we show that by using our non-uniform sampling strategies, the factors appearing in such error recursion enjoy a much better dependence on problem specific quantities, e.g.,  such as the condition number (Table~\ref{tab: comp}). For example, using block partial leverage score sampling, the factor for the linear term of the error recursion~\eqref{eq:rate0_inexact} is of order $\mathcal{O}(\sqrt{\kappa})$ as opposed to $\mathcal{O}(\kappa)$ for uniform sampling.
\item We show that to achieve a locally {\it problem independent} linear convergence rate, i.e., $\|\w_{t+1} - \w^\ast \| \leq \rho \|\w_t - \w^\ast \|$ for some fixed constant $\rho < 1$, the per-iteration complexities of our algorithm with leverage scores sampling and block norm squares sampling are $\tilde \bigO(\nnz(\A)\log n + d^2\kappa^{3/2})$ and $\tilde \bigO(\nnz(\A) + {\bf sr}(\A)d\kappa^{5/2})$, respectively, which have lower dependence on condition numbers compared to~\cite{agarwal2016second, pilanci2015newton,roosta2016sub2} (Table~\ref{tab: complexity}). In particular, 
in the presence of high non-uniformity among $\{\A_{i}(\w)\}_{i=1}^{n}$, factors $\bar \kappa$ and $\hat \kappa$ (see Definition~\eqref{eq:cond_nums}) which appear in SSN (uniform)~\cite{roosta2016sub1}, and LiSSA~\cite{agarwal2016second}, can potentially be as large as $\Omega(n \kappa)$;  see Section~\ref{sec:comparison} for detailed discussions.
\item We numerically demonstrate the effectiveness and robustness of our algorithms in recovering the minimizer of ridge logistic regression on several real datasets with a moderately large condition number (Figures~\ref{fig: lambdas},~\ref{fig:app}~and~\ref{fig: logistic}). In particular, our algorithms are at least twice as fast as Newton's methods in finding a medium- or high-precision solution, while other methods including first-order methods converge slowly. 

\end{compactitem}

\vspace{2mm}
\noindent
The remainder of the paper is organized as follows. We begin in Section~\ref{sec: background} with notation and assumptions that will be used throughout the paper. In Section~\ref{sec: main}, we describe our main algorithm and propose two non-uniform sampling schemes. Section~\ref{sec: theory} provides theoretical analysis. 
Finally, we present our numerical experiments in Section~\ref{sec: experiments}. 

\section{Background}\label{sec: background}


\subsection{Notation} \label{sec:notation}
Given a function $F$, the gradient, the exact Hessian and the approximate Hessian are denoted by $\g$, $\H$, and $\widetilde \H$, respectively. Iteration counter is denoted by subscript, e.g., $\w_{t}$. Unless stated specifically, $\|\cdot\|$ denotes the Euclidean norm for vectors and spectral norm for matrices. Frobenius norm of matrices is written as $\|\cdot\|_{F}$.
Given a matrix $\A$, we let $\nnz(\A)$ denote the number of non-zero elements in $\A$.
By a matrix $\A$ having $n$ blocks, we mean that $\A$ has a block structure and can be viewed as $\A = \begin{pmatrix} \A_{1}^{T}  \cdots \A_{n}^{T} \end{pmatrix}^{T}$, for appropriate size blocks $\A_{i}$.

\begin{definition}[Tangent Cone]
Denote $\K$ be the tangent cone of constraints $\C$ at the optimum $\w^\ast$, i.e., $\K = \{ \Delta \vert \w^\ast + t\Delta \in \C \text{ for some } t>0\}$.
\end{definition}

\begin{definition}[$\K$-restricted Maximum and Minimum Eigenvalues]
\label{def:k_eig}
Given a symmetric matrix $\A$ and a cone $\K$, we define the $\K$-restricted maximum and minimum eigenvalues as follows.
\begin{equation*}
  \lmin^\K (\A) = \min_{\x \in \K \setminus \{\mathbf{0}\}} \frac{\x^T \A \x}{\x^T \x}, ~~~\lmax^\K (\A) = \max_{\x \in \K \setminus \{\mathbf{0}\}} \frac{\x^T \A \x}{\x^T \x}.
\end{equation*}
\end{definition}

\begin{definition}[Stable Rank]
\label{def: sr}
Given a matrix $\A \in \reals^{N\times d}$, the stable rank of $A$ is defined as 
$${\bf sr}(\A) = \dfrac{\|\A\|_F^2}{\|\A\|_2^2}.$$
\end{definition}


\begin{definition}[Leverage Scores]
\label{def:lev}
 Given $\A\in \bbR^{n\times d}$, then for $i=1,\ldots,n$, the $i$-th leverage scores of $\A$ is defined as
\begin{equation*}
\label{eq:def_lev}
\tau_i (\A) = \a_i^T(\A^T\A)^\dagger \a_i.
\end{equation*}
\end{definition}


\subsection{Assumptions}
\label{sec:assumptions}
Throughout the paper, we use the following assumptions regarding the properties of the problem.
\begin{assumption}[Lipschitz Continuity]
\label{assump:1}
$F(\w)$ is convex and twice differentiable. The Hessian is $L$-Lipschitz continuous, i.e.
\begin{equation*}
\label{eq:lip}
\|\nabla^2 F(\u) - \nabla^2 F(\v) \|\le L \|\u- \v\|, \quad \forall \u, \v \in \C.
\end{equation*}
\end{assumption}

\begin{assumption}[Local Regularity]
\label{assump:2}
$F(\x)$ is locally strongly convex and smooth, i.e.,
 \begin{equation*}
\label{eq:loc_reg}
\mu = \lmin^\K (\nabla^2 F(\w^\ast)) > 0, \quad \nu = \lmax^\K (\nabla^2 F(\w^\ast)) < \infty.
\end{equation*}
Here we define the local condition number of the problem as $\kappa: = \nu/\mu$.
\end{assumption}

\begin{assumption}[Hessian Decomposition]
\label{assump:pos}
For each $f_i(\w)$ in~\eqref{eq:p}, define $\nabla^2 f_i(\w) \defeq \H_i(\w) \defeq \A_i^T(\w) \A_i(\w)$. For simplicity, we assume $k_1 = \cdots = k_n = k$ and $k$ is independent of $d$. Furthermore, we assume that given $\w$, computing $\A_i(\w)$, $\H_i(\w)$, and $\g(\w)$ takes $\bigO(d)$, $\bigO(d^2)$, and $\bigO(\nnz(\A))$ time, respectively. We call the matrix $\A(\w) = \begin{pmatrix} \A_1^T, \ldots, \A_n^T \end{pmatrix}^T \in \reals^{nk\times d}$ the augmented matrix of $\{\A_i(\w)\}$. Note that $\H(\w) = \A(\w)^T \A(\w) + \Q(\w)$.
\end{assumption}



\section{Main Algorithm: SSN with Non-uniform Sampling}\label{sec: main}



%

Our proposed SSN method with non-uniform sampling is given in Algorithm~\ref{alg:main}. The core of our algorithm is based on choosing a sampling scheme $\mathcal{S}$ that, at every iteration, constructs a non-uniform sampling distribution $\{p_i\}_{i=1}^n$ over $\{\A_i(\w_t)\}_{i=1}^n$ and then samples from $\{\A_i(\w_t)\}_{i=1}^n$ to form the approximate Hessian, $\widetilde \H(\w_t)$. The sampling sizes $s$ needed for different sampling distributions will be discussed in Sections~\ref{sec:lev} and \ref{sec:rnorms}. Since $\H(\w) = \sum_{i=1}^n \A^T_i(\w) \A_i(\w) + \Q(\w)$, the Hessian approximation essentially boils down to a matrix approximation problem. Here, we generalize the two popular non-uniform sampling strategies, i.e., leverage score sampling and block norm squares sampling, which are commonly used in the field of randomized linear algebra, particularly for matrix approximation problems \cite{holodnak2015randomized,mahoney2011randomized}. 
With an approximate Hessian constructed via non-uniform sampling, we may choose an appropriate solver $\mathcal A$ to the solve the sub-problem in Step 11 of Algorithm~\ref{alg:main}. 
Below we elaborate on the construction of the two non-uniform sampling schemes. 
Indeed, the sampling distribution is defined based on the matrix representation of $\{\H_i(\w_t)\}_{i=1}^n$ --- its augmented matrix defined as follows.
\begin{definition}[Augmented Matrix]
\label{def:aug}
Define the augmented matrix of $\{\H_i(\w_t)\}_{i=1}^n$ as
$$ \A(\w) = \begin{pmatrix} \A_1^T & \cdots & \A_n^T
 \end{pmatrix}^T \in \reals^{kn\times d}. $$
\end{definition}
\noindent
For the ease of presentation, throughout the rest of this section and next section, we use $\A$ and $\Q$ to denote $\A(\w)$ and $\Q(\w)$, respectively, as long as it is clear in the text.

\begin{algorithm}[ht]
\caption{Sub-sampled Newton method(SSN) with Non-uniform Sampling}
\label{alg:main}
\begin{algorithmic}[1]
\STATE {\bf Input:} Initialization point $\w_0$, number of iteration $T$, sampling scheme $\mathcal{S}$ and solver $\mathcal{A}$.
\STATE {\bf Output:} $\w_{T}$
\FOR{$t = 0,\ldots,T-1$}
\STATE Construct the non-uniform sampling distribution $\{p_i\}_{i=1}^n$ as described in Section~\ref{sec: main}. 
\FOR{$i = 1,\ldots,n$}
\STATE $q_i = \min\{s\cdot p_i, 1\}$, where $s$ is the sampling size.
\STATE 
   $ \widetilde \A_i(\w_t) =
  \begin{cases} \A_i(\w_t) / \sqrt{q_i}, & \text{with probability } q_i, \\
   \mathbf{0}, & \text{with probability } 1 - q_i.
  \end{cases} $
\ENDFOR
\STATE $\widetilde \H(\w_t) = \sum_{i=1}^n \widetilde \A_i^T(\w_t) \widetilde \A_i(\w_t) + \Q(\w_t)$.
\STATE Compute $\g(\w_t)$
\STATE Use solver $\mathcal{A}$ to solve the sub-problem inexactly
\vspace{-1mm}
\begin{equation}
\label{eq:sub}
\vspace{-2mm}
\w_{t+1} \approx \arg\min_{\w\in \mathcal C} \{\frac{1}{2}\langle (\w - \w_t), \widetilde \H(\w_t) (\w - \w_t)\rangle + \langle \g(\w_t), \w-\w_t\rangle\}.
\end{equation}
\ENDFOR
\RETURN $\w_{T}$.
\end{algorithmic}
\end{algorithm}



\paragraph{Block Norm Squares Sampling}
The first option is to construct a sampling distribution based on the magnitude of $\A_i$. That is, define
\begin{equation}
\label{eq:dist_row_norms}
  p_i = \frac{\|\A_i\|_F^2}{\|\A\|_F^2}, \quad i = 1,\ldots,n.
\end{equation}
This is an extension to the row norm squares sampling in which the intuition is to capture the importance of the blocks based on the ``magnitudes'' of the sub-Hessians. 

\paragraph{Block Partial Leverage Scores Sampling}
The second option is to construct a sampling distribution based on leverage scores.
Compared to the traditional matrix approximation problem, this problem has two major difficulties. First, here blocks are being sampled, not single rows. Second, the matrix being approximated involves not only $\A$ but also $\Q$.

To address the first difficulty, we follow the work by \citet{harvey_sparse} in which a sparse sum of semi-definite matrices is found by sampling based on the trace of each semi-definite matrix after a proper transformation. By expressing $\A^T \A = \sum_{i=1}^n \A_i^T \A_i$, one can show that their approach is essentially sampling based on the sum of leverage scores that correspond to each block.   
For the second difficulty, inspired by the recently proposed ridge leverage scores~\cite{el2014fast, cohen2015ridge}, we consider the leverage scores of a matrix that concatenates $\A$ and $\Q^\frac{1}{2}$.
Combining these motivates us to define a new notion of leverage scores ---- block partial leverage scores which is define as follows formally.
\begin{definition}[Block Partial Leverage Scores]
\label{def:bplev}
Given a matrix $\A \in \reals^{kn\times d}$ with $n$ blocks and a matrix $\Q \in \reals^{d\times d}$ satisfying $\Q \succeq \mathbf{0}$, let $\{\tau_i\}_{i=1}^{kn+d}$ be the leverage scores of the matrix $\begin{pmatrix} \A \\ \Q^\frac{1}{2}\end{pmatrix}$. Define the block partial leverage score for the $i$-th block as
 \begin{equation*}
     \tau^\Q_i(\A) = \sum_{j=k(i-1)+1}^{ki} \tau_j.
 \end{equation*}
\end{definition}

\noindent
Then the sampling distribution is defined as
\begin{equation}
\label{eq:dist_lev}
  p_i = \frac{\tau_i^\Q(\A)}{\sum_{j=1}^n \tau_j^\Q(\A)}, \quad i = 1,\ldots,n.
\end{equation}

\noindent
{\bf Remark.} 
When each block of $\A$ has only one row and $\Q = \mathbf{0}$, the partial block leverage scores are equivalent to the ordinary leverage scores defined in Definition~\ref{def:lev}.



\section{Theoretical Results}\label{sec: theory}

In this section we provide detailed theoretical analysis to describe the complexity of our algorithm.\footnote{In this work, we only focus on local convergence guarantees for Algorithm~\ref{alg:main}. To ensure global convergence, one can incorporate an existing globally convergent method, e.g.~\cite{roosta2016sub2}, as initial phase and switch to Algorithm~\ref{alg:main} once the iterate is ``close enough'' to the optimum; see Lemma~\ref{lem:general}.} Different choices of sampling scheme $\mathcal{S}$ and the sub-problem solver $\mathcal{A}$ lead to different complexities in SSN. More precisely, total complexity is characterized by the following four factors: {\bf (i)} total number of iterations $T$ determined by the convergence rate which is affected by the choice of $\mathcal{S}$ and $\mathcal{A}$;  {\bf (ii)} the time, $t_{grad}$, it takes to compute the full gradient $\g(\w_t)$ (Step 10 in Algorithm~\ref{alg:main}), {\bf (iii)} the time $t_{const}$, to construct the sampling distribution $\{p_i\}_{i=1}^n$ and sample $s$ terms at each iteration (Steps 4-8 in Algorithm~\ref{alg:main}), which is determined by $\mathcal{S}$;  and {\bf (iv)} the time $t_{solve}$ needed to (implicitly) form $\tilde \H$ and (inexactly) solve the sub-problem at each iteration (Steps 9 and 11 in Algorithm~\ref{alg:main}) which is affected by the choices of both $\mathcal{S}$ (manifested in the sampling size $s$) and $\mathcal{A}$. With these, the total complexity can be expressed as
\begin{equation}
 \label{eq:complexity}
  T \cdot (t_{grad} + t_{const} + t_{solve}).
\end{equation}

Below we study these contributing factors.
Lemma~\ref{lem:general} in Section~\ref{subsec: conditions} gives a general structural lemma that characterize the convergence rate which determines $T$.
In Sections~\ref{sec:lev} and \ref{sec:rnorms}, Lemmas~\ref{lem:lev_const} and \ref{lem:rnorms_const} discuss $t_{const}$ for the two sampling schemes respectively while Lemmas~\ref{thm:lev} and \ref{thm:rnorms} give the required sampling size $s$ for the two sampling schemes respectively which directly affects the $t_{solve}$. Furthermore, $t_{solve}$ is also affected by the choice of solver which will be discussed in Section~\ref{sec:solvers}. Finally, the complexity results are summarized in Section~\ref{sec:complexity} and a comparison with other methods is provided in Section~\ref{sec:comparison}.


\subsection{Sufficient conditions for local linear-quadratic convergence}
\label{subsec: conditions}
Before diving into details of the complexity analysis, we state a structural lemma that characterizes the local convergence rate of our main algorithm, i.e., Algorithm~\ref{alg:main}.
As discussed earlier, there are two layers of approximation in Algorithm~\ref{alg:main}, i.e., approximation of the Hessian by sub-sampling and inexactness of solving~\eqref{eq:sub}. For the first layer, we require the approximate Hessian to satisfy one of the following two conditions (in Sections~\ref{sec:lev} and \ref{sec:rnorms} we shall see our construction of approximate Hessian via non-uniform sampling can achieve these conditions with a sampling size independent of $n$).
\begin{equation}
\label{eq:cond1}
\| \widetilde \H(\w_t)  - \H(\w_t) \| \le \epsilon \cdot \|\H(\w_t)\|, \tag{\bf C1}
\end{equation}
or
\begin{equation}
\label{eq:cond2}
|\x^T (\widetilde \H(\w_t)  - \H(\w_t))\y| \le \epsilon \cdot \sqrt{\x^T \H(\w_t) \x}\cdot \sqrt{\y^T \H(\w_t) \y}, ~~\forall \x,\y\in \K. \tag{\bf C2}
\end{equation}
Note that \eqref{eq:cond1} and \eqref{eq:cond2} are two commonly seen guarantees for matrix approximation problems. In particular, \eqref{eq:cond2} is stronger in the sense that the spectrum of the approximated matrix $\H(\w_t)$ is well preserved. Below in Lemma~\ref{lem:general}, we shall see such a stronger condition ensures a better dependence on the condition number in terms of the convergence rate.
For the second layer of approximation, we require the solver to produce an $\epsilon_0$-approximate solution $\w_{t+1}$ satisfying
\begin{equation}
 \label{eq:err_req}
   \|\w_{t+1} - \w_{t+1}^\ast \| \leq \epsilon_0 \cdot \|\w_t - \w_{t+1}^\ast\|,
\end{equation}
where $\w_{t+1}^\ast$ is the exact optimal solution to~\eqref{eq:sub}. Note that~\eqref{eq:err_req} implies an $\epsilon_{0}$-relative error approximation to the exact update direction, i.e., $\| \v - \v^\ast \| \leq \epsilon \|\v^\ast\|$ where $\v = \w_{t+1} - \w_{t}, \: \v^\ast = \w_{t+1}^\ast - \w_t$.

\noindent
{\bf Remark.}
When the problem is unconstrained, i.e., $\C = \reals^d$, solving the subproblem~\eqref{eq:sub} is equivalent to solving
\begin{equation*}
  \tilde \H_t \v = -\nabla F(\w_t).
\end{equation*}
Then requirement~\eqref{eq:err_req} is equivalent to finding an approximation solution $\v$ such that
\begin{equation*}
  \| \v - \v^\ast \| \leq \epsilon \|\v^\ast\|.
\end{equation*}

\begin{lemma}[Structural Result]
\label{lem:general}
Let $\{\w_t\}_{i=1}^T$ be the sequence generated based on update rule~\eqref{eq:sub} with initial point $\w_0$ satisfying $\|\w_0 - \w^\ast \| \leq \frac{\mu}{4L}$. 
Under Assumptions $\ref{assump:1}$ and $\ref{assump:2}$, if condition~\eqref{eq:cond1} or \eqref{eq:cond2} is met, we have the following results.
\begin{itemize}
\item If the subproblem is solved exactly,
then the solution error satisfies the following recursion
\begin{equation}
\label{eq:rate0}
\|\w_{t+1} - \w^*\| \le C_q \cdot \| \w_t - \w^*\|^2 + C_l \cdot \|\w_t - \w^*\|,
\end{equation}
where $C_q$ and $C_l$ are specified in \eqref{eq:rate1} or \eqref{eq:rate2} below.
\item If the subproblem is solved approximately and $\w_{t+1}$ satisfies~\eqref{eq:err_req}, then the solution error satisfies the following recursion
\begin{equation}
\label{eq:rate0_inexact}
\|\w_{t+1} - \w^*\| \le (1+\epsilon_0)C_q \cdot \| \w_t - \w^*\|^2 + (\epsilon_0 + (1+\epsilon_0)C_l) \cdot \|\w_t - \w^*\|,
\end{equation}
where $C_q$ and $C_l$ are specified in \eqref{eq:rate1} or \eqref{eq:rate2} below.
\end{itemize}
Specifically, given any $\epsilon \in (0,1/2)$, 
\begin{itemize}
\item
If the approximate Hessian $\tilde \H_t$ satisfies~\eqref{eq:cond1}, then in \eqref{eq:rate0} and \eqref{eq:rate0_inexact}
\begin{equation}
\label{eq:rate1}
 C_q = \frac{2L}{(1- 2\epsilon \kappa) \mu}, \quad C_l = \frac{4\epsilon \kappa}{1-2\epsilon \kappa}.
\end{equation}
\item
If the approximate Hessian $\tilde \H_t$ satisfies~\eqref{eq:cond2}, then in \eqref{eq:rate0} and \eqref{eq:rate0_inexact}
\begin{equation}
\label{eq:rate2}
 C_q = \frac{2L}{(1- \epsilon) \mu}, \quad C_l =\frac{3\epsilon }{1-\epsilon}\sqrt{\kappa}.
\end{equation}
\end{itemize}
\end{lemma}

We remark that Lemma~\ref{lem:general} is applicable to $F(\w_t)$ and $\tilde \H_t$ of any form.
In our case, specifically,
since $\nabla^2 F(\w_t) = \A^T \A + \Q$ and $\tilde \H_t = \A^T \S^T \S \A + \Q$ where $\S$ is the resulting sampling matrix.
\eqref{eq:cond1} is equivalent
\begin{equation}
   \label{eq:cond1p}
  \| (\A^T \S^T \S \A + \Q) - (\A^T \A + \Q) \| \leq \epsilon \|\A^T \A + \Q\|, 
\end{equation}
and due to Lemma~\ref{lem:equv_cond} in Appendix~\ref{sec:bplev_results}, \eqref{eq:cond2} is equivalent to
\begin{equation} 
 \label{eq:cond2p}
  -\epsilon (\A^T\A+\Q) \preceq (\A^T \S^T \S \A + \Q) - (\A^T \A + \Q) \preceq \epsilon (\A^T \A + \Q). 
\end{equation}
From this it is not hard to see that \eqref{eq:cond2} is strictly stronger than \eqref{eq:cond1}.
Also, in this case the Hessian approximation problem boils down to a matrix approximation problem. That is, given $\A$ and $\Q$, we want to construct a sampling matrix $\S$ efficiently such that the matrix $\A^T \A + \Q$ is well preserved.
As we mentioned, leverage scores sampling and block norm squares sampling are two popular ways for this task. In the next two subsections we will focus on the theoretical properties of these two schemes.

\subsection{Results for block partial leverage scores sampling}
\label{sec:lev}

\subsubsection{Construction}
\label{sec:lev_const}
Since the block partial leverage scores are defined as the standard leverage scores of some matrix, we can make use of the fast approximation algorithm for standard leverage scores. Specifically, apply a variant of the algorithm in~\citep{drineas2012fast} by using the sparse subspace embedding~\cite{clarkson13sparse} as the underlying sketching method to further speed up the computation.

\begin{theorem}
\label{lem:lev_const}
Given $\w$, under Assumption~\ref{assump:pos}, with high probability, it takes $t_{const} = \bigO(\nnz(\A) \log n)$ time to construct a set of approximate leverage scores $\{\hat\tau^\Q_i(\A)\}_{i=1}^n$ that satisfy $\tau^\Q_i(\A) \leq \hat\tau^\Q_i(\A) \leq \beta \cdot \tau^\Q_i(\A)$ where $\{\tau_i\}_{i=1}^n$ are the block partial leverage scores of $\H(\w) = \sum_{i=1}^n \H_i(\w) + \Q(\w)$ where $\A$ is the augmented matrix of $\{ \H_i(\w) \}_{i=1}^n$, and $\beta$ is a constant.
\end{theorem}

\subsubsection{Sampling size}
The following theorem indicates that if we sample the blocks of $\A$ based on block partial leverage scores with large enough sampling size, \eqref{eq:cond2p} holds with high probability.
\begin{theorem}
\label{thm:lev}
Given $\A$ with $n$ blocks, $\Q \succeq \mathbf{0}$ and $\epsilon \in (0,1)$, let $\{ \tau^\Q_i(\A) \}_{i=1}^n$ be its block partial leverage scores and $\{ \hat\tau^\Q_i(\A) \}_{i=1}^n$ be their overestimates, i.e., $ \hat\tau^\Q_i(\A) \ge \tau^\Q_i(\A), i = 1,...,n$.  Let $p_i = \frac{ \hat\tau^\Q_i(\A)}{\sum_{j=1}^n \hat\tau^\Q_j(\A)}$. Construct $\S\A$ by sampling the $i$-th block of $\A$ with probability $q_i = \min\{s\cdot p_i,1\}$ and rescaling it by $1/\sqrt{q_i}$. Then if 
\begin{equation}
  s \geq 4 \left( \sum_{i=1}^n \hat\tau_i^\Q(\A) \right) \cdot \log\frac{4d}{\delta} \cdot \frac{1}{\epsilon^2},
\end{equation}
with probability at least $1-\delta$, \eqref{eq:cond2p} holds, thus \eqref{eq:cond2} holds.
\end{theorem}

\noindent
{\bf Remark.}
When $\{ \tau^\Q_i(\A) \}_{i=1}^n$ are the exact scores, since $\sum_{i=1}^n \tau_i^\Q(\A) \leq \sum_{i=1}^{N+d} \tau_i (\bar \A) = d$ where $\bar \A = \begin{pmatrix} \A \\ \Q^\frac{1}{2}\end{pmatrix}$, the above theorem indicates that less than $\bigO(d\log d/\epsilon^2)$ blocks are needed for \eqref{eq:cond2p} to hold.


\subsection{Results for block norm squares sampling}
\label{sec:rnorms}

\subsubsection{Construction}
\label{sec:rnorms_const}

To sample based on block norm squares, one has first compute the Frobenius norm of every block in the augmented matrix $\A$. This requires $\bigO(\nnz(\A))$ time.

\begin{theorem}
\label{lem:rnorms_const}
Given $\w$, under Assumption~\ref{assump:pos}, it takes $t_{const} = \bigO(\nnz(\A))$ time to construct a block norm squares sampling distribution for $\H(\w) = \sum_{i=1}^n \H_i(\w) + \Q(\w)$ where $\A$ is the augmented matrix of $\{ \H_i(\w) \}_{i=1}^n$.

\end{theorem}

\subsubsection{Sampling size} The following theorem \citep{holodnak2015randomized} show the approximation error bound for Gram matrix. Here we extend it to our augmented matrix setting as follows,
\begin{theorem}[\cite{holodnak2015randomized}]
\label{thm:rnorms}
Given $\A$ with $n$ blocks, $\Q \succeq \mathbf{0}$ and $\epsilon \in (0,1)$, for $i=1,\ldots,n$, let $r_i = \|\A_i\|_F^2$. Let $p_i = \frac{r_i}{\sum_{j=1}^n r_j}$. Construct $\S\A$ by sampling the $i$-th block of $\A$ with probability $q_i = \min\{s\cdot p_i,1\}$ and rescaling it by $1/\sqrt{q_i}$. Then if 
\begin{equation}
  s \geq 4 {\bf sr}(\A) \cdot \log\frac{\min\{4{\bf sr}(\A),d\}}{\delta} \cdot \frac{1}{\epsilon^2},
\end{equation}
with probability at least $1-\delta$, \eqref{eq:cond1p} holds, thus \eqref{eq:cond1}.
\end{theorem}


\subsection{Discussion on the choice of solver}
\label{sec:solvers}
Here we discuss the effect of the choice of the solver $\mathcal{A}$ in Algorithm~\ref{alg:main}. Specifically, after an approximate Hessian $\tilde \H_t$ is constructed in Algorithm~\ref{alg:main}, we look at how various solvers affect $t_{solve}$ in~\eqref{eq:complexity}.
Since the approximate Hessian $\tilde \H_t$ is of the form $\A^T \S^T \S\A + \Q$ where $\S\A \in \reals^{s\times d}$, the complexity for solving the subproblem~\eqref{eq:sub} essentially depends on $s$ and $d$.
Given $\S\A$ and $\Q$, for ease of notation, we use
\begin{equation*}
 t_{solve} = \mathcal{T}(\mathcal{A},\C,s,d)
\end{equation*}
to denote the time it needs to solve the subproblem~\eqref{eq:sub} using solver $\mathcal{A}$.

For example, when the problem is unconstrained, i.e., $\C = \reals^d$, the subproblem~\eqref{eq:sub} reduces to a linear regression problem with size $s$ by $d$ and direct solver costs $\bigO(sd^2)$. Alternatively, one can use an iterative solver such as Conjugate Gradient (CG) to obtain an approximate solution. 
In this case, the complexity for solving the subproblem becomes $\bigO(sd\sqrt{\tilde\kappa_t}(\log \frac{1}{\epsilon_0} + \log \tilde{\kappa_t}))$ to produce an $\epsilon_0$ solution to~\eqref{eq:sub}  where $\tilde \kappa_t$ is the condition number of the problem.
It is not hard to see that CG is advantageous when the low dimension $d$ is large and the linear system is fairly well-conditioned.


There are also many solvers that are suitable. In Table~\ref{tab:solver}, we give a few examples for the unconstrained case ($\C = \reals^d$) by summarizing the complexity and the resulting approximation quality $\epsilon_0$ in~\eqref{eq:err_req}.

\begin{table}[H]
\centering
\begin{tabular}{c|cc|c}
\sc
solver $\mathcal A$ & $\mathcal{T}(\mathcal{A},\reals^d,s,d)$ & $\epsilon_0$ & reference
\\
\hline
direct & $\bigO(sd^2)$ & $0$ & NA \\
CG & $\bigO(sd\sqrt{\tilde\kappa_t} \log(1/\epsilon))$ & $\sqrt{\tilde\kappa_t} \epsilon$ & \citep{golub2012matrix}
\\
GD& $\bigO(sd \tilde\kappa_t \log(1/\epsilon))$&  $\epsilon$ & \citep[Theorem 2.1.15]{nesterov2013introductory}
\\
ACDM & $\bigO(s {\bf sr}(\S\A)\sqrt{\tilde\kappa_t}\log(1/\epsilon))$& $\sqrt{\tilde\kappa_t }\epsilon$ & \citep{lee2013efficient}
\end{tabular}
\caption{Comparison of different solvers for the subproblem. Here $\tilde\kappa_t = \lmax^\K(\tilde\H_t)/\lmin^\K(\tilde\H_t)$.}
\label{tab:solver}
\end{table}


\subsection{Complexities}
\label{sec:complexity}
Again, recall that in \eqref{eq:complexity} the complexity of the sub-sampled Newton methods can be expressed as $T \cdot (t_{const} + t_{grad} + t_{solve})$.
Combining the results from the previous few subsections, we have the following lemma characterizing the total complexity.

\begin{theorem}
\label{thm:complexity}
For Algorithm~\ref{alg:main} with sampling scheme $\mathcal{S}$ and solver $\mathcal{A}$, the total complexity is
\begin{equation*}
   T \cdot (t_{const} + \mathcal{T}(\mathcal{A},\C,s,d)),
\end{equation*}
and the solution error is specified in Lemma~\ref{lem:general}.
In the above, $t_{const}$ is specified in Theorem~\ref{lem:lev_const} and Theorem~\ref{lem:rnorms_const} and $s$ is specified in Theorem~\ref{thm:lev} and Theorem~\ref{thm:rnorms} depending on the choice of $\mathcal{S}$; $\mathcal{T}(\mathcal{A},\C,s,d)$ is discussed in Section~\ref{sec:solvers}.
\end{theorem}


Indeed, Lemma~\ref{lem:general} implies that the sub-sampled Newton method inherits a local constant linear convergence rate. This can be shown by choosing specific values for $\epsilon$ and $\epsilon_0$ in Lemma~\ref{lem:general}. The results are presented in the following corollary.

\begin{corollary}
\label{cor:complexity}
Suppose $\C = \reals^d$. In Algorithm~\ref{alg:main}, assume that $CG$ is used to solve the subproblem~\eqref{eq:sub}. Then under Assumption~\ref{assump:pos}, 
\begin{itemize}
\item
if block partial leverage scores sampling is used, the complexity per iteration in the local phase is
\begin{equation}
\tilde\bigO(\nnz(\A)\log n + d^2\kappa^{3/2});
\end{equation}
\item
if block norm squares sampling is used, the complexity per iteration in the local phase is
\begin{equation}
\tilde\bigO(\nnz(\A) + {\bf sr}(\A)d\kappa^{5/2}),
\end{equation}
\end{itemize}
and the solution error satisfies
\begin{equation}
 \label{eq:linear_rate}
     \|\w_{t+1} - \w^\ast \| \leq \rho \cdot \|\w_t - \w^\ast \|,
\end{equation} 
where $\rho \in(0, 1)$ is a constant.\footnote{In this paper, $\tilde\bigO(\cdot)$ hides logarithmic factors of $d$, $\kappa$ and $1/\delta$.}
\end{corollary}


\subsection{Comparisons}
\label{sec:comparison}

\subsubsection{Comparison between different sampling schemes}
\label{sec:unif_comp}
As discussed above, the sampling scheme $\mathcal{S}$ plays a crucial role in sub-sampled Newton methods. Here, we compare the two proposed non-uniform sampling schemes, namely, block partial leverage scores sampling and block norm squares sampling, with uniform sampling. SSN with uniform sampling was discussed in~\citep{roosta2016sub1}. For completeness, we state the sampling size bound for uniform sampling. Note that, this upper bound for $s$ is tighter than the original analysis in~\citep{roosta2016sub1}.

\begin{theorem}
\label{thm:unif}
Given $\A$ with $n$ blocks, $\Q \succeq \mathbf{0}$ and $\epsilon \in (0,1)$, construct $\S\A$ by uniform sampling $s$ blocks from $\A$ and rescaling it by $\sqrt{n/s}$. Then if 
\begin{equation}
 \label{eq:s_unif}
  s \geq 4n \cdot \frac{\max_i \|\A_i\|^2}{\|\A\|^2} \cdot \log \frac{d}{\delta} \cdot \frac{1}{\epsilon^2},
\end{equation}
with probability at least $1-\delta$, \eqref{eq:cond1p} holds, thus \eqref{eq:cond1} holds.
\end{theorem}

This result allows us to compare the three sampling schemes in terms of the three main complexities, i.e., $t_{const}$, $t_{solve}$ and $T$ (manifested in $C_q$ and $C_l$), as shown in Table~\ref{tab: comp} (identical to Table~\ref{tab: results} in Section~\ref{sec:intro}).
Note here, to evaluate the effect of the sampling scheme $\mathcal{S}$ only, we assume a direct solver is used for the subproblem~\eqref{eq:sub} since in this case $t_{solve}$ is directly controlled by the sampling size $s$, independent of the solver $\mathcal{A}$.
Also, for simplicity, we assume that $\mathcal{C} = \reals^d$. The analysis is similar for general cases.
In Table~\ref{tab: comp}, $C_q$ and $C_l$ are defined based on two problem properties $\kappa$ and $\tilde \kappa$:
\begin{equation}
\label{eq:constants}
\tilde \kappa = L/\mu, \quad \kappa = \nu/\mu,
\end{equation}
where constants $L, \mu, \nu$ are defined in Assumptions~\ref{assump:1} and~\ref{assump:2}. Also, throughout this subsection, for randomized algorithms, we choose parameters such that the failure probability is a constant.

\begin{tableprime}{tab: results}
\centering
\begin{tabular}{c|c|c|cc}
\sc 
Name & $t_{const}$ & $t_{solve} = sd^2$ & $C_q$ & $C_l$   \\
\hline
Newton's method & $0$ & $\bigO(nd^2)$ & $\tilde \kappa$ & $0$ \\
SSN (leverage scores) & $\bigO(\nnz(\A) \log n)$ & $\tilde \bigO((\sum_i\tau_i^\Q(\A)) d^2/\epsilon^2)$ & $\frac{\tilde \kappa}{1-\epsilon}$ & $\frac{\epsilon\sqrt{\kappa}}{1-\epsilon} $    \\
SSN (block norm squares) & $\bigO(\nnz(\A))$ & $\tilde \bigO({\bf sr}(\A)d^2/\epsilon^2)$ & $\frac{\tilde \kappa}{1-\epsilon\kappa} $ & $\frac{\epsilon\kappa}{1-\epsilon\kappa}$  \\
SSN (uniform) & $\bigO(1)$ &  $\tilde\bigO\left(nd^2\frac{\max_i \|\A_i\|^2}{ \|\A\|^2}/\epsilon^2\right)$ & $\frac{\tilde \kappa}{1-\epsilon\kappa} $ & $\frac{\epsilon\kappa}{1-\epsilon\kappa}$  
\end{tabular}
\caption{Comparison between standard Newton's method and sub-sampled Newton methods (SSN) with different sampling schemes. In the above, $C_q$ and $C_l$ are the constants achieved in~\eqref{eq:rate0}; $\kappa$ and $\tilde \kappa$ are defined in~\eqref{eq:constants}; $\A \in \reals^{\bigO(n)\times d}$ is the augmented matrix in the current iteration (Definition~\ref{def:aug}) that satisfies $\A^T \A = \sum_{i=1}^n \H_i(\w_t)$;
${\bf sr}(\A)$ is the stable rank of $\A$ satisfying ${\bf sr}(\A) \leq d$; $\nnz(\A)$ denote the number of non-zero elements in $\A$.
Note here, to remove the effect of solver $\mathcal{A}$, we assume the subproblem~\eqref{eq:sub0} is solved exactly. Also, we assume the problem is unconstrained ($\C = \reals^d$) so that $t_{solve} = sd^2$.
}
\label{tab: comp}
\end{tableprime}
  
As can be seen in Table~\ref{tab: comp}, the greatest advantage of uniform sampling scheme comes from its simplicity of construction. On the other hand, as discussed in Sections~\ref{sec:lev_const} and \ref{sec:rnorms_const}, it takes nearly input sparsity time to construct the leverage scores sampling distribution or the block norm squares sampling distribution.
When it comes to the sampling size $s$ for achieving \eqref{eq:cond1p} or \eqref{eq:cond2p}, as suggested in~\eqref{eq:s_unif}, the one for uniform sampling can become $\Omega(n)$ when $\A$ is very non-uniform, i.e., $\max_i \|\A_i\| \approxeq \|\A\|$. It can be shown that for a given $\epsilon$, block norm squares sampling requires the smallest sampling size which leads to the smallest value of $t_{solve}$ in Table~\ref{tab: comp}.

It is worth pointing that, although either \eqref{eq:cond1p} or \eqref{eq:cond2p} is sufficient to yield a local linear-quadratic convergence rate, as \eqref{eq:cond2p} is essentially a stronger condition, it has better constants, i.e., $C_q$ and $C_l$.
This fact is reflected in Table~\ref{tab: comp}. The constants $C_q$ and $C_l$ for leverage scores sampling have a better dependence on the local condition number $\kappa$ than the other two schemes since leverage scores sampling yields a sampling matrix that satisfies the spectral approximation guarantee~\eqref{eq:cond2p}. 
In fact, this difference can dramatically affect the performance of the algorithm when dealing with ill-conditioned problems.
This is verified by numerical experiments; see Figure~\ref{fig: lambdas} in Section~\ref{sec: experiments} for details.

\noindent
{\bf Remark.} Note that all the analysis including error recursion (Lemma~\ref{lem:general}) and the required sampling size $s$ for different sampling schemes are provided as upper bounds. There will be cases that the sampling size bound indicates a large value for $s$, in fact a much smaller sampling size $s$ suffices to yield good performance. 
For example, when the leverage scores are equal or close to a uniform distribution, the actual required sampling size for uniform sampling scheme is much less than~\eqref{eq:s_unif}. 


\subsubsection{Comparison between various methods}
\label{sec:complexity_comp}

Next, we compare our main algorithm with other stochastic second-order methods including~\cite{roosta2016sub1, agarwal2016second}. Since these essentially imply a constant linear convergence rate, i.e.,
\begin{equation}
 \label{eq:const_lin}
     \|\w_{t+1} - \w^\ast \| \leq \rho \cdot \|\w_t - \w^\ast \|,  ~~0 < \rho < 1,
\end{equation} 
we compare the complexity per iteration needed in each algorithm when such a rate~\eqref{eq:const_lin} is desired.
Note here, for the ease of comparison, we assume $\C = \reals^d$, $R(\w) = 0$, and CG is used for solving sub-problems in SSN so that the complexities can be easily expressed. This is the same setting as in \cite{agarwal2016second}.\footnote{In \cite{agarwal2016second}, the authors also considered the ridge penalty term but they absorb the penalty term into the summation which still makes the objective in the form of an average/sum of functions.} Analysis for general cases is similar.

Note that the results in the related works are stated in terms of condition numbers that are defined differently from the local condition number (Assumption~\ref{assump:2}) used in this paper.
To be precise, besides the standard definition, i.e., $\kappa$, for any $\w \in \reals^d$, define
\begin{subequations}
\label{eq:cond_nums}
\begin{eqnarray}
\kappa(\w) &=& \dfrac{\lmax(\sum_{i=1}^n \H_i(\w))}{\lmin(\sum_{i=1}^n \H_i(\w))}, \\
\hat\kappa(\w) &=& n \cdot \frac{\max_i \lmax(\H_i(\w))}{\lmin(\sum_{i=1}^n \H_i(\w))}, \\
\bar \kappa(\w) &=& \frac{\max_i \lmax(\H_i(\w))}{\min_i \lmin(\H_i(\w))}.
\end{eqnarray}
\end{subequations}


An immediate relationship between the three condition numbers is 
$
 \kappa(\w) \le \hat\kappa(\w)\le \bar \kappa(\w).
$
The connections between these condition numbers depend on the properties of $\H_i(\w)$. Roughly speaking, when all $\H_i(\w)$'s are ``close'' to each other, then $\lmax^\K(\sum_{i=1}^n \H_i(\w)) \approx \sum_{i=1}^n \lmax^\K(\H_i(\w)) \approx n \cdot \max_i \lmax^\K(\H_i(\w))$, and thus $\kappa \approx \hat\kappa$. And similarly, $\kappa \approx \bar\kappa$. While in many cases, some $\H_i(\w)$'s can be very different from the rest. For example, when solving linear regression, the Hessian $\H(\w) = \A^T\A$, where $\A$ is the data matrix with each row as a data point. When the rows are not very uniform, it can be the case that $\kappa$ is smaller than $\hat\kappa$ and $\bar \kappa$ by a a factor of $n$.


\begin{tableprime}{tab: complexity}
\centering
\begin{tabular}{c|c|c}
\sc 
name & \sc complexity per iteration & \sc reference \\
\hline
Newton-CG method & $\tilde \bigO(\nnz(\A)\sqrt{\kappa})$ & \cite{nocedal2006numerical} \\
SSN (leverage scores) & $\tilde \bigO(\nnz(\A)\log n + d^2\kappa^{3/2})$ & {\bf This paper} \\
SSN (block norm squares) &  $\tilde \bigO(\nnz(\A) + {\bf sr}(\A)d\kappa^{5/2} )$ & {\bf This paper} \\
Newton Sketch (SRHT) & $\tilde\bigO(nd (\log n)^4 + d^2(\log n)^4 \kappa^{3/2})$ & \cite{pilanci2015newton} 
\\
SSN (uniform) & $\tilde\bigO(\nnz(\A) + d\hat\kappa \kappa^{3/2})$ & \cite{roosta2016sub2} \\
LiSSA & $\tilde\bigO(\nnz(\A) + d \hat\kappa \bar \kappa^2)$ & \cite{agarwal2016second}
\end{tabular}
\caption{
Complexity per iteration of different methods to obtain a problem independent local linear convergence rate. The quantities $\kappa$, $\hat \kappa$, and $\bar \kappa$ are the local condition numbers, defined in~\eqref{eq:cond_nums} at the optimum $\w^\ast$, satisfying 
$\kappa \leq \hat\kappa \leq \bar\kappa$; $\A \in \reals^{\bigO(n)\times d}$ is the augmented matrix in the current iteration (Definition~\ref{def:aug}) that satisfies $\A^T \A = \sum_{i=1}^n \H_i(\w_t)$; ${\bf sr}(\A)$ is the stable rank of $\A$ satisfying ${\bf sr}(\A) \leq d$; $\nnz(\A)$ denote the number of non-zero elements in $\A$.
Note here, for the ease of comparison, we assume $\C = \reals^d$, $R(\w) = 0$, and CG is used for solving sub-problems in SSN so that the complexities can be easily expressed.
}
\label{tab: complexity2}
\end{tableprime}

Given the notation we defined, we summarize the complexities of different algorithms in Table~\ref{tab: complexity2} (identical to Table~\ref{tab: complexity} in Section~\ref{sec:intro}) including Newton's methods with CG solving the subproblem.
One immediate conclusion we can draw is that compared to Newton's methods, these stochastic second-order methods trade the coefficient of the leading term $\bigO(nd)$ with some lower order terms that only depend on $d$ and condition numbers (assuming $\nnz(\A) \approx nd$). Therefore, one should expect these algorithm to perform well when $n \gg d$ and the problem is fairly well-conditioned.

Although SSN with non-uniform sampling has a quadratic dependence on $d$, its dependence on the condition number is better than the other methods. There are two main reasons.
First, the total power of the condition number is lower, regardless the versions of the condition number needed.
Second, SSN (leverage scores) and SSN (block norm squares) only depend on $\kappa$ which can be significantly lower than the other two definitions of condition number according to the discussion above.
Overall, SSN (leverage scores) is more robust to the ill-conditioned problems.

\section{Numerical Experiments}\label{sec: experiments}


We consider an estimation problem in GLMs with Gaussian prior. Assume $\X\in \reals^{n\times d}, \Y \in \mathcal{Y}^n$ are the data matrix and response vector. The problem of minimizing the negative log-likelihood with ridge penalty can be written as 
\vspace{-5mm}
\begin{equation}
\min_{\w\in \mathcal \reals^d}  \: \sum_{i=1}^n \psi(\x_i^T\w, y_i) + \lambda \|\w\|_2^2,\nonumber
\end{equation}
where $\psi: \reals \times \mathcal Y \rightarrow \reals $ is a convex cumulant generating function and $\lambda \ge 0$ is the ridge penalty parameter.
In this case, the Hessian is $\H(\w) = \sum_{i=1}^n \psi^{''}(\x_i^T\w, y_i) \x_i \x_i^T + \lambda \mathbf{I} \defeq \X^T \D^2(\w) \X + \lambda \mathbf{I}$, where $\x_i$ is $i$-th column of $\X^T$ and $\D(\w)$ is a diagonal matrix with the diagonal $[\D(\w)]_{ii} = \sqrt {\psi^{''}(\x_i^T\w, y_i)}$.
 The augmented matrix of $\{\A_i(\w)\}$ can be written as $\A(\w) = \D \X\in \reals^{n\times d}$ where $\A_i(\w) = [\D(\w)]_{ii} \x_i^T$. 

For our numerical simulations, we consider a very popular instance of GLMs, namely, logistic regression, where $\psi(u,y) = \log( 1 + \exp(-uy))$ and $\mathcal{Y} = \{\pm 1\}$. 
Table~\ref{tab:data} summarizes the datasets used in our experiments.
\begin{table}[H]
\vspace{-2mm}
\centering
\label{tab:data}
\begin{tabular}{c|c|c|c|c}
\sc dataset &  \tt CT slices\tablefootnote{\url{https://archive.ics.uci.edu/ml/datasets/Relative+location+of+CT+slices+on+axial+axis}} & \tt Forest\tablefootnote{\url{https://archive.ics.uci.edu/ml/datasets/Covertype}} & \tt Adult\tablefootnote{\url{https://archive.ics.uci.edu/ml/datasets/Adult}} & \tt Buzz\tablefootnote{\url{https://archive.ics.uci.edu/ml/datasets/Buzz+in+social+media+}} \\
\hline
$n$ & 53,500 & 581,012 & 32,561 & 59,535\\
$d$ & 385 & 55 & 123 & 78\\
$\kappa$ & 368 & 221 & 182 & 37 \\ 
$\hat \kappa$ & 47,078 & 322,370 & 69,359 & 384,580 \\
\end{tabular}
\caption{Datasets used in ridge logistic regression. In the above, $\kappa$ and $\bar \kappa$ are the local condition numbers of ridge logistic regression problem with $\lambda = 0.01$ as defined in~\eqref{eq:cond_nums}.
}

\end{table}
\vspace{-3mm}
We compare the performance of the following five algorithms:
(i) {\it Newton}: the standard Newton's method, (ii) {\it{Uniform}}:  SSN with uniform sampling,
(iii) \textit{PLevSS}: SSN with partial leverage scores sampling,
(iv) \textit{RNormSS}: SSN with block (row) norm squares sampling, and
(v) \textit{LBFGS-k}: standard L-BFGS method~\citep{liu1989lbfgs} with history size $k$,
(vi) \textit{GD}: Gradient Descent,
(vii) \textit{AGD}: Accelerated Gradient Descent(AGD)~\cite{nesterov2013introductory}.
Note that, despite all of our effort, we could not compare with methods introduced in~\cite{agarwal2016second, erdogdu2015convergence} as they seem to diverge in our experiments.

All algorithms are initialized with a zero vector.\footnote{Theoretically, the suitable initial point for all the algorithms is the one with which the standard Newton's method converges with a unit stepsize. Here, $\w_{0} = \mathbf{0}$ happens to be one such good starting point.} We also use CG to solve the sub-problem approximately to within $10^{-6}$ relative residue error.
%
In order to compute the relative error $\|\w_{t} - \w^*\|/\|\w^*\|$, an estimate of $\w^\ast$ is obtained by running the standard Newton's method for sufficiently long time.
Note here, in SSN with partial leverage score sampling, we recompute the leverage scores every $10$ iterations. 
Roughly speaking, these ``stale'' leverage scores can be viewed as approximate leverage scores for the current iteration with approximation quality that can be upper bounded by the change of the Hessian and such quantity is often small in practice.
So reusing the leverage scores allows us to further drive down the running time.

We first investigate the effect of the condition number, controlled by varying $\lambda$, on the performance of different methods, and the results are depicted in Figure~\ref{fig: lambdas}. It can be seen that in well-conditioned cases,  all sampling schemes work equally well. However, as the condition number gets larger, the performance of uniform sampling deteriorates, while non-uniform sampling, in particular leverage score sampling, shows a great degree of robustness to such ill-conditioning effect. 
The experiments shown in Figure~\ref{fig: lambdas} are consistent with the theoretical results of Table~\ref{tab: comp}.

\begin{figure}[htb]
\vspace{-3mm}
  \centering
\subfigure[condition number]{
    \includegraphics[width=0.3\textwidth]{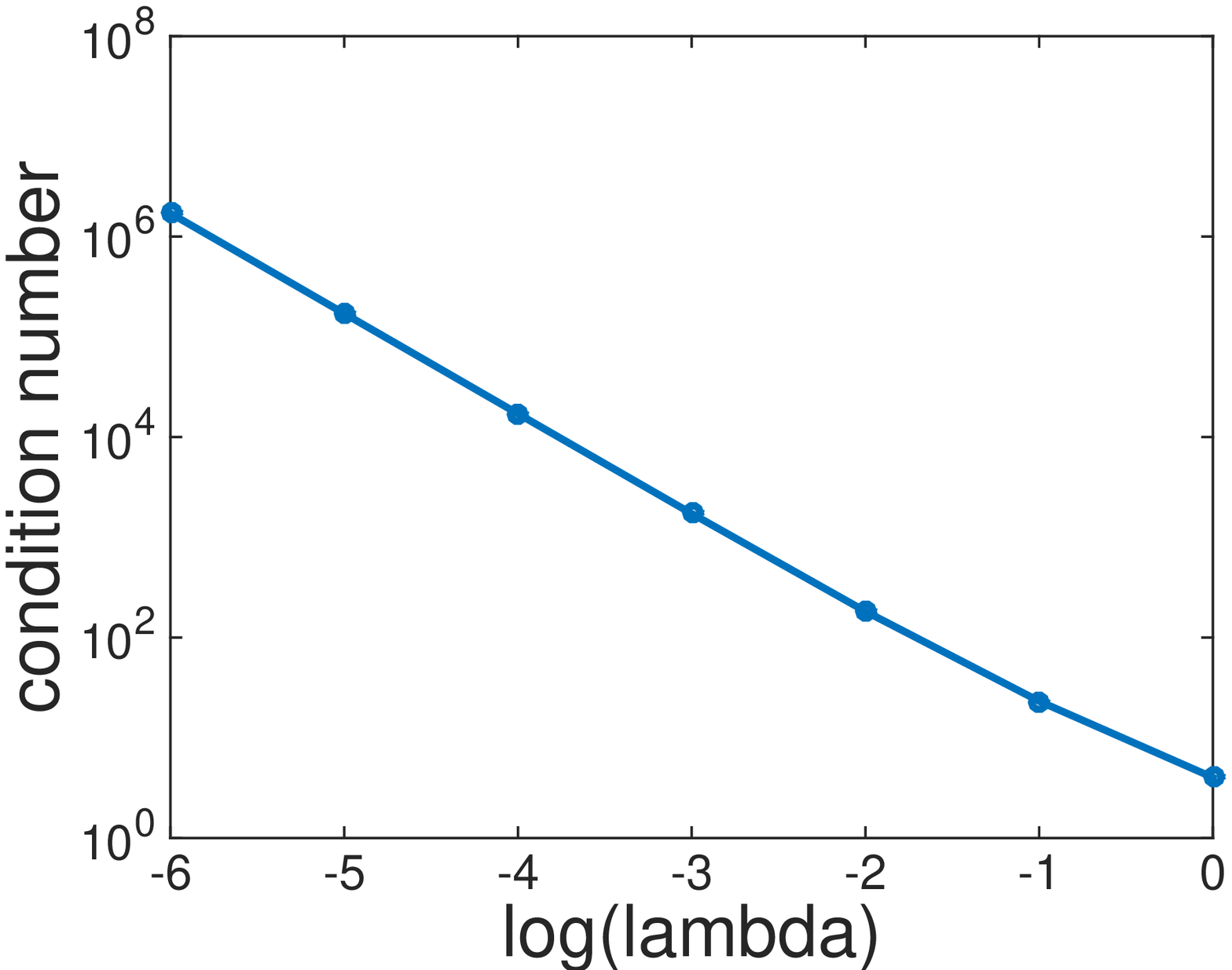}}
\subfigure[sampling size]{
    \includegraphics[width=0.3\textwidth]{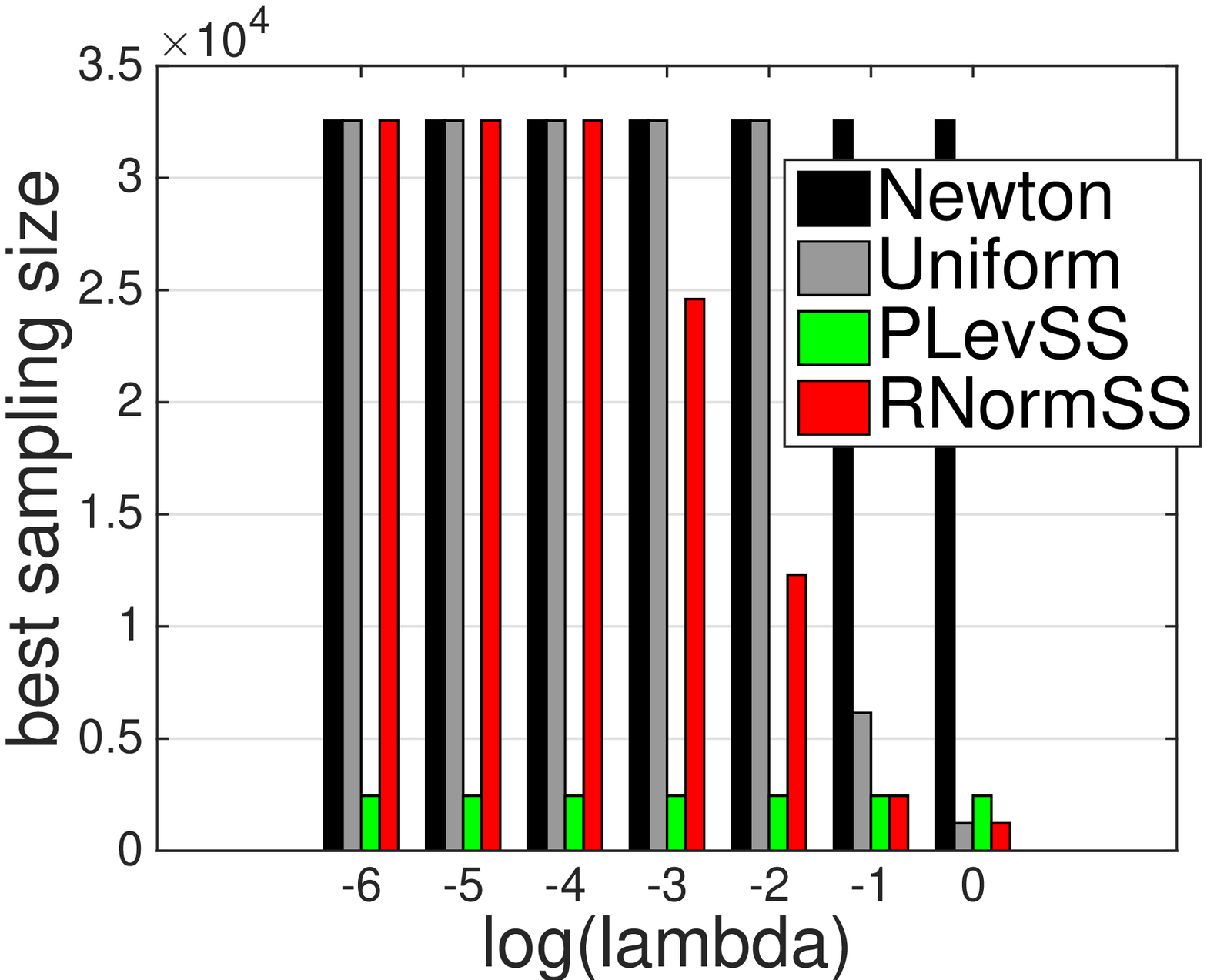}}
\subfigure[running time]{
    \includegraphics[width=0.3\textwidth]{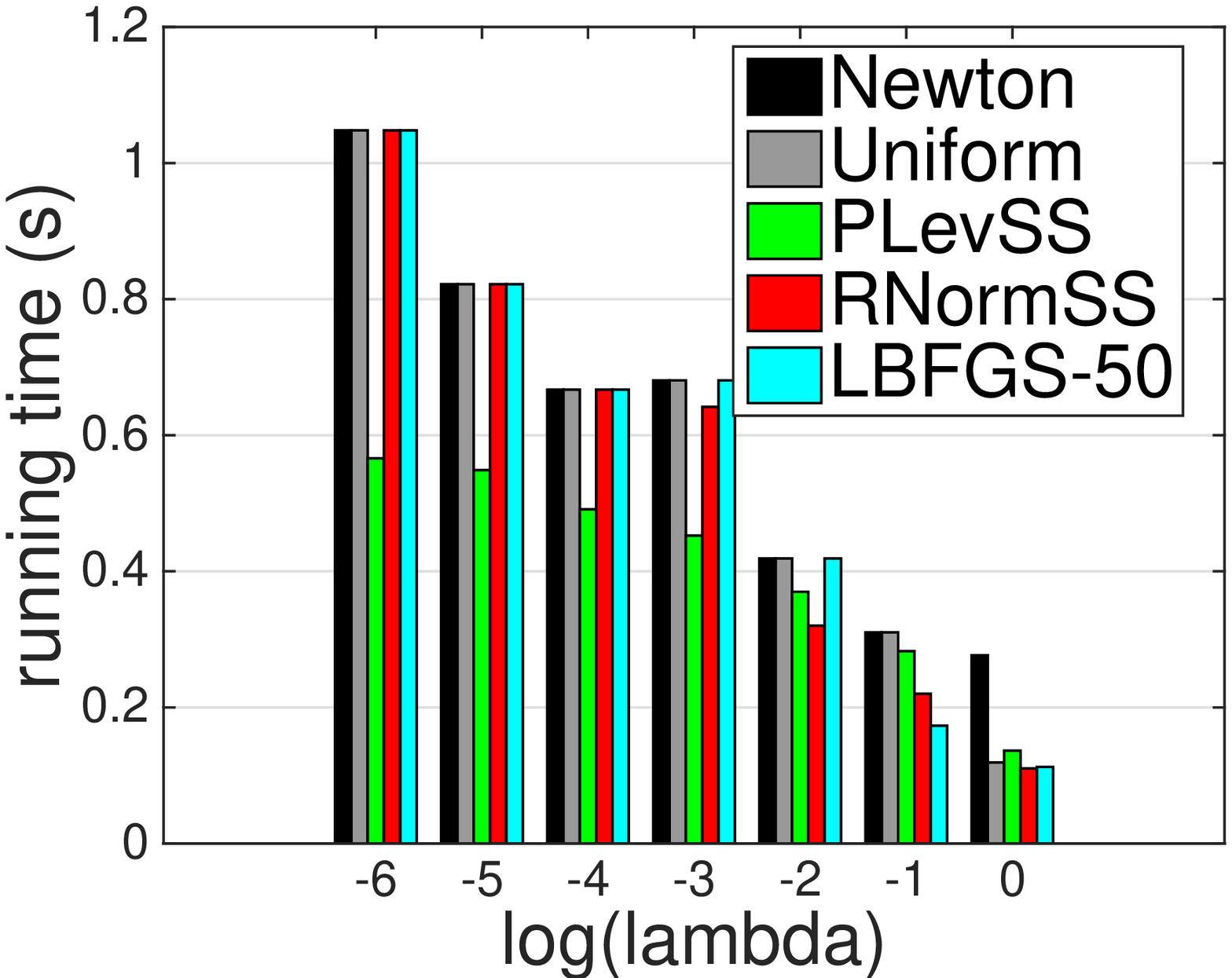}}
\caption{Ridge logistic regression on \texttt{Adult} with different $\lambda$'s: (a) local condition number $\kappa$, (b) sample size for different SSN methods giving the best overall running time, (c) running time for different methods to achieve $10^{-8}$ relative error. 
}
\label{fig: lambdas}
\end{figure}

Next, we compare the performance of various methods as measured by relative-error of the solution vs.\ running time.
First, we provide a set of empirical comparison between first-order and second-order methods in Figure~\ref{fig:app}.\footnote{For each sub-sampled Newton method, the sampling size is determined by choosing the best value from $\{10d,20d,30d,...,100d, 200d,300d,...,1000d\}$ in the sense that the objective value drops to $1/3$ of initial function value first. }
This is on dataset \texttt{CT Slice} with two different $\lambda$'s.
As can be seen clearly in Figure~\ref{fig:app}, SSN with non-uniform sampling not only drives down the loss function $F(\w)$ to an arbitrary precision much more quickly, but also recovers the minimizer $\w^\ast$ to a high precision while first-order methods such as Gradient Descent converge very slowly.
More importantly, unlike SSN with uniform sampling and LBFGS, non-uniform SSN exhibits a better robustness to condition number as it performance doesn't deteriorate much when the problem becomes more ill-conditioned (by setting the regularization $\lambda$ smaller in Figures~\ref{fig:app}(c) and Figure~\ref{fig:app}(d)). This robustness to condition number allows our approach to excel for a wider range of models.

\begin{figure}[h!tbp]
\centering
\subfigure[\scriptsize $\lambda = 10^{-2}$, error in $\w$]{
    \includegraphics[width=0.4\textwidth]{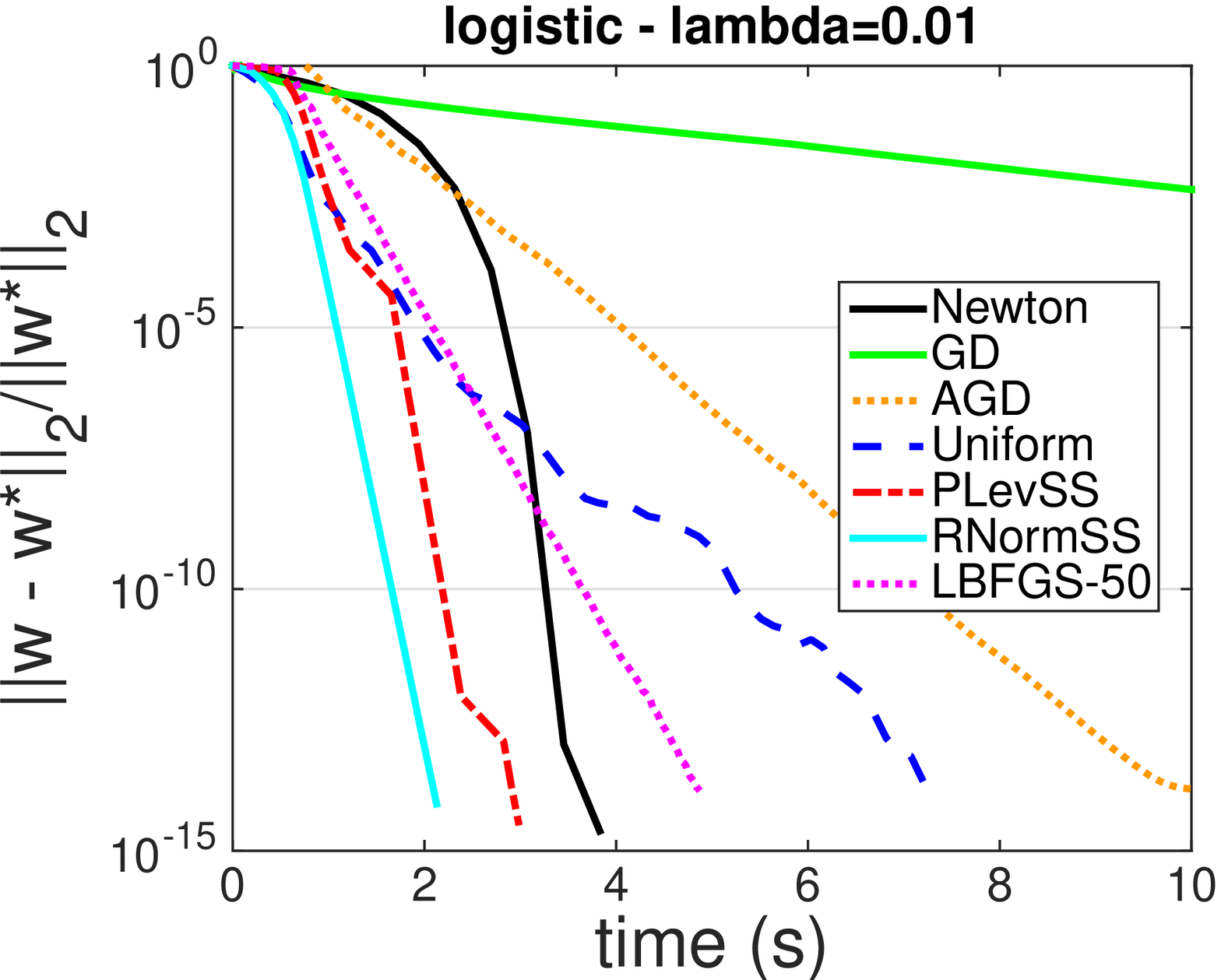}}
\subfigure[\scriptsize $\lambda = 10^{-2}$, error in $F(\w)$]{
    \includegraphics[width=0.4\textwidth]{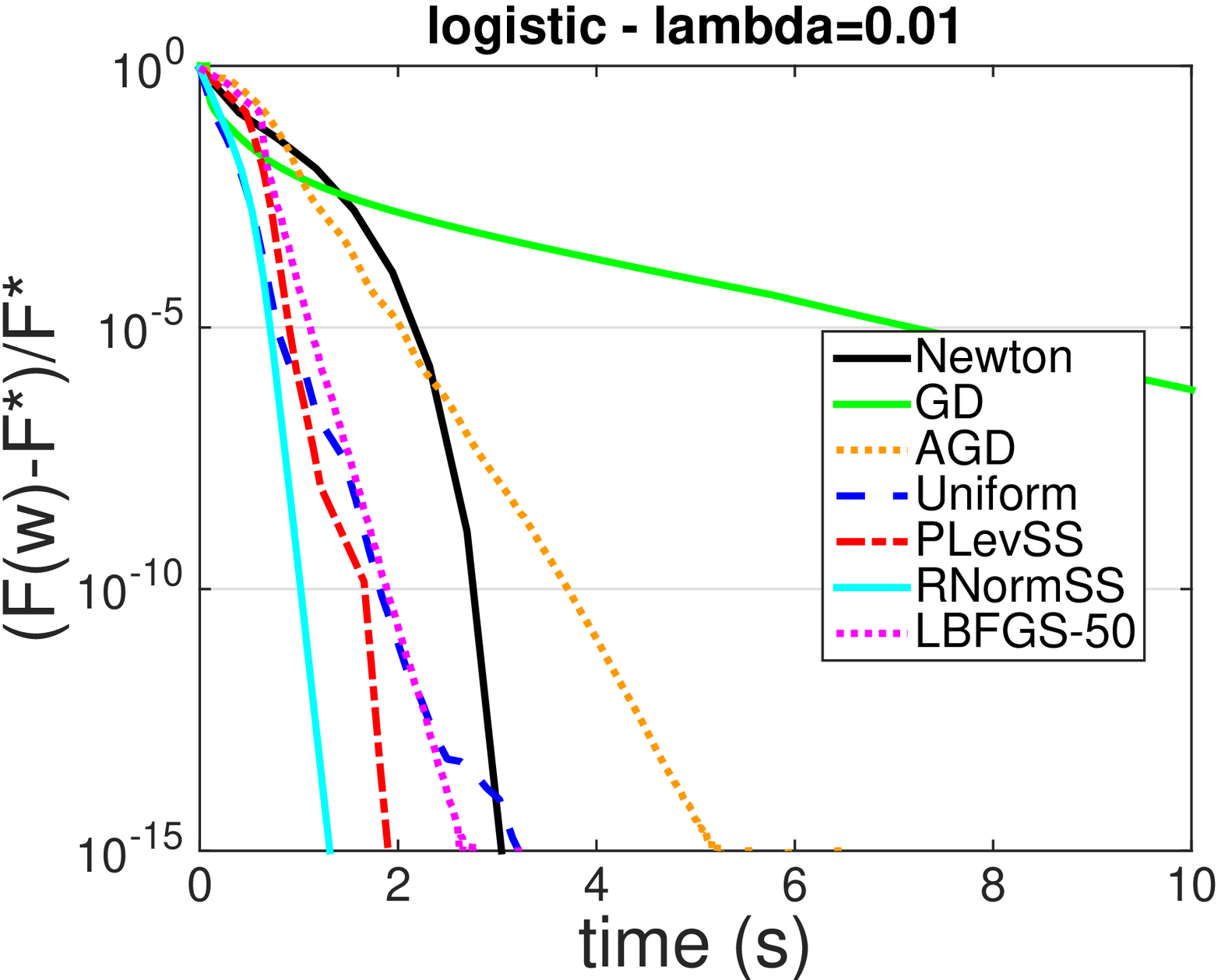}} \\
\subfigure[\scriptsize $\lambda = 10^{-4}$, error in $\w$]{
    \includegraphics[width=0.4\textwidth]{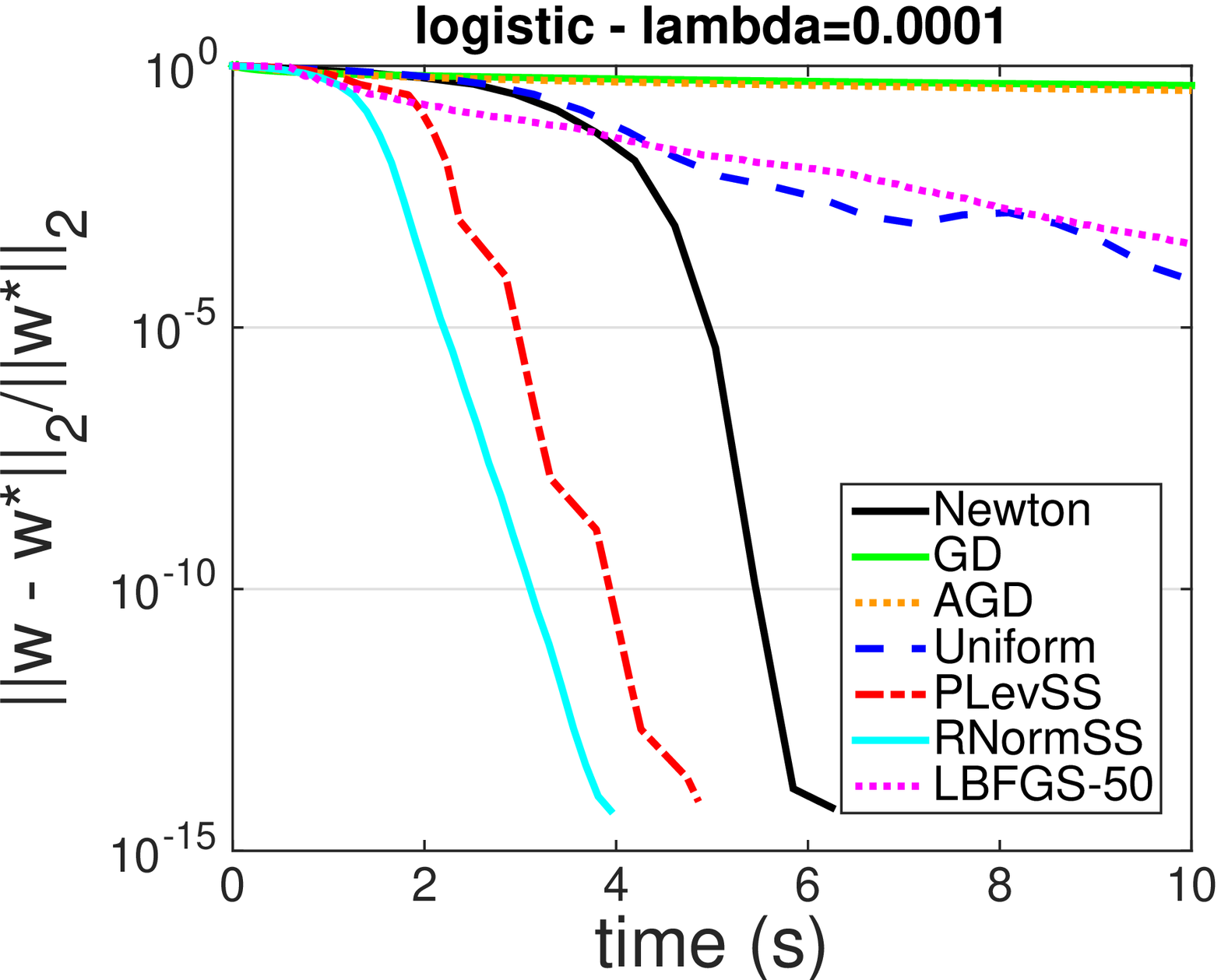}} 
  \subfigure[\scriptsize $\lambda = 10^{-4}$, error in $F(\w)$]{
    \includegraphics[width=0.4\textwidth]{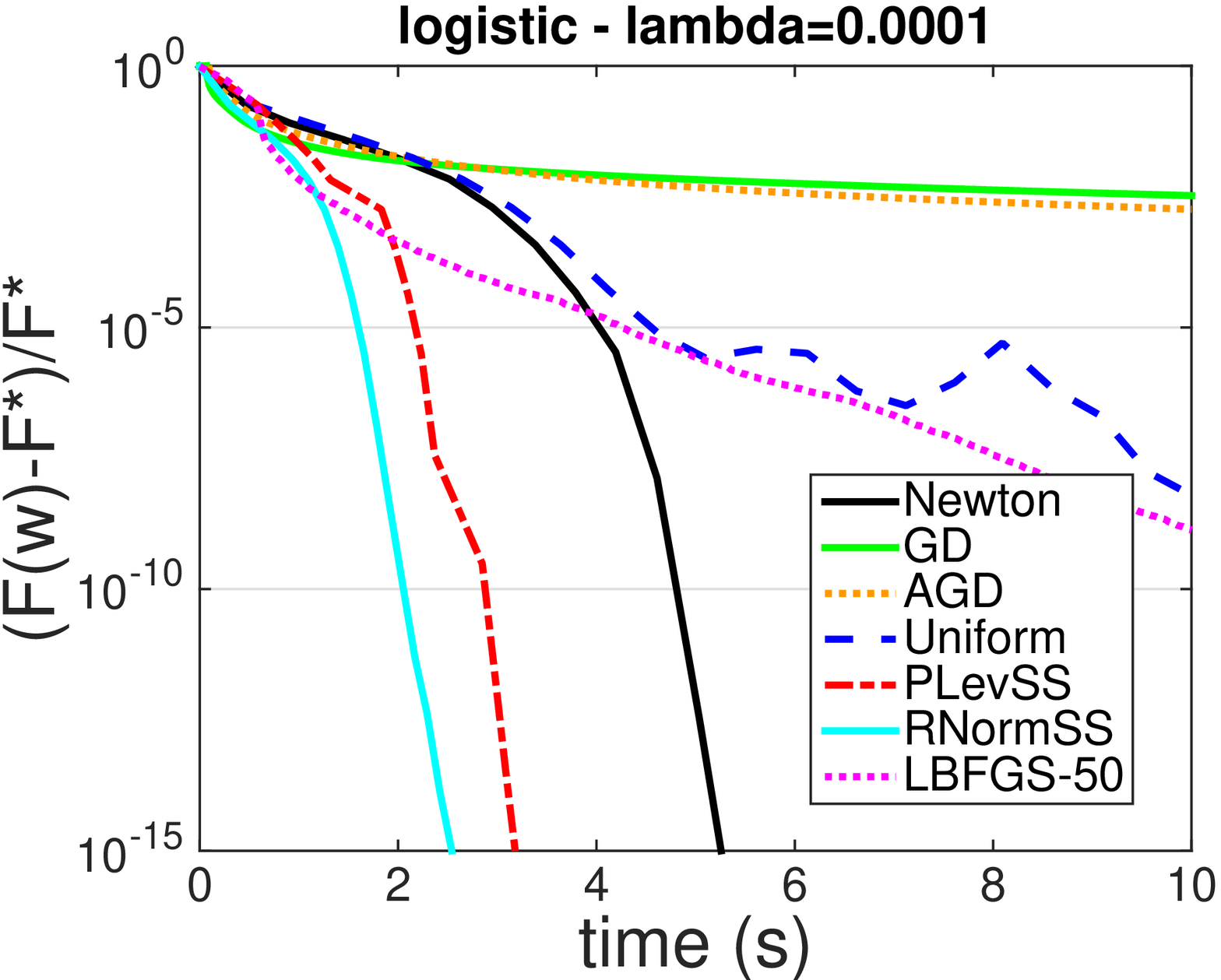}}
  \caption{Iterate relative error vs.\ time(s) for a ridge logistic regression problem with two choices of regularization parameter $\lambda$ on a real dataset \texttt{CT Slice}. Various second-order methods including standard Newton, LBFGS, SSN with uniform sampling (Uniform), partial leverage scores sampling (PLevSS) and row norm squares sampling (RNormSS), as well as gradient descent (GD) and its accelerated version (AGD) as representatives of first-order methods are implemented. Here, when $\lambda = 10^{-2}$, $\kappa = 386$, when $\lambda = 10^{-4}$, $\kappa = 1.387\times 10^4$.}
  \label{fig:app}
\end{figure}

A more comprehensive comparison among various second-order methods on the four datasets is presented in Figure~\ref{fig: logistic}.
It can be seen that, in most cases, SSN with non-uniform sampling schemes, i.e., PLevSS and RNormSS, outperform the other algorithms, especially Newton's method. In particular, they can be as twice faster as Newton's method. This is because on datasets with large $n$, the computational gain of our sub-sampled Newton methods in forming the (approximate) Hessian is significant while their convergence rate is only slightly worse than Newton's method (not shown here).
Moreover, recall that in Section~\ref{sec:comparison} we discussed that the convergence rate of SSN with uniform sampling relies on $\bar \kappa$.
When the problem exhibits a high non-uniformity among data points, i.e., $\bar \kappa$ is much higher than $\kappa$ as shown in Table~\ref{tab:data},
uniform sampling scheme performs poorly, e.g., in Figure~\ref{fig: logistic}(b).

\begin{figure}[h!tbp]
  \centering
\subfigure[\texttt{CT Slice}]{
    \includegraphics[width=0.4\textwidth]{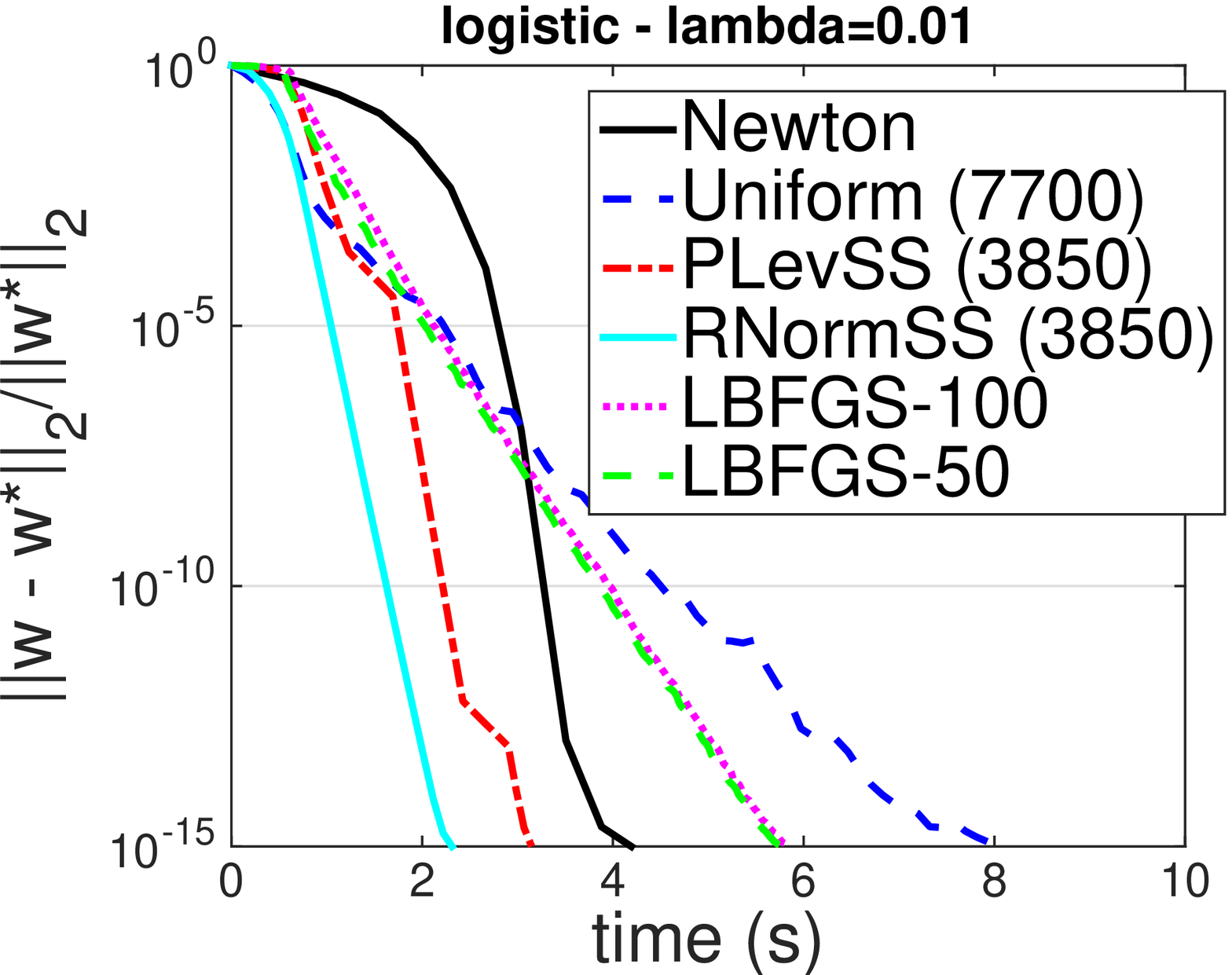}}
\subfigure[\texttt{Forest}]{
    \includegraphics[width=0.4\textwidth]{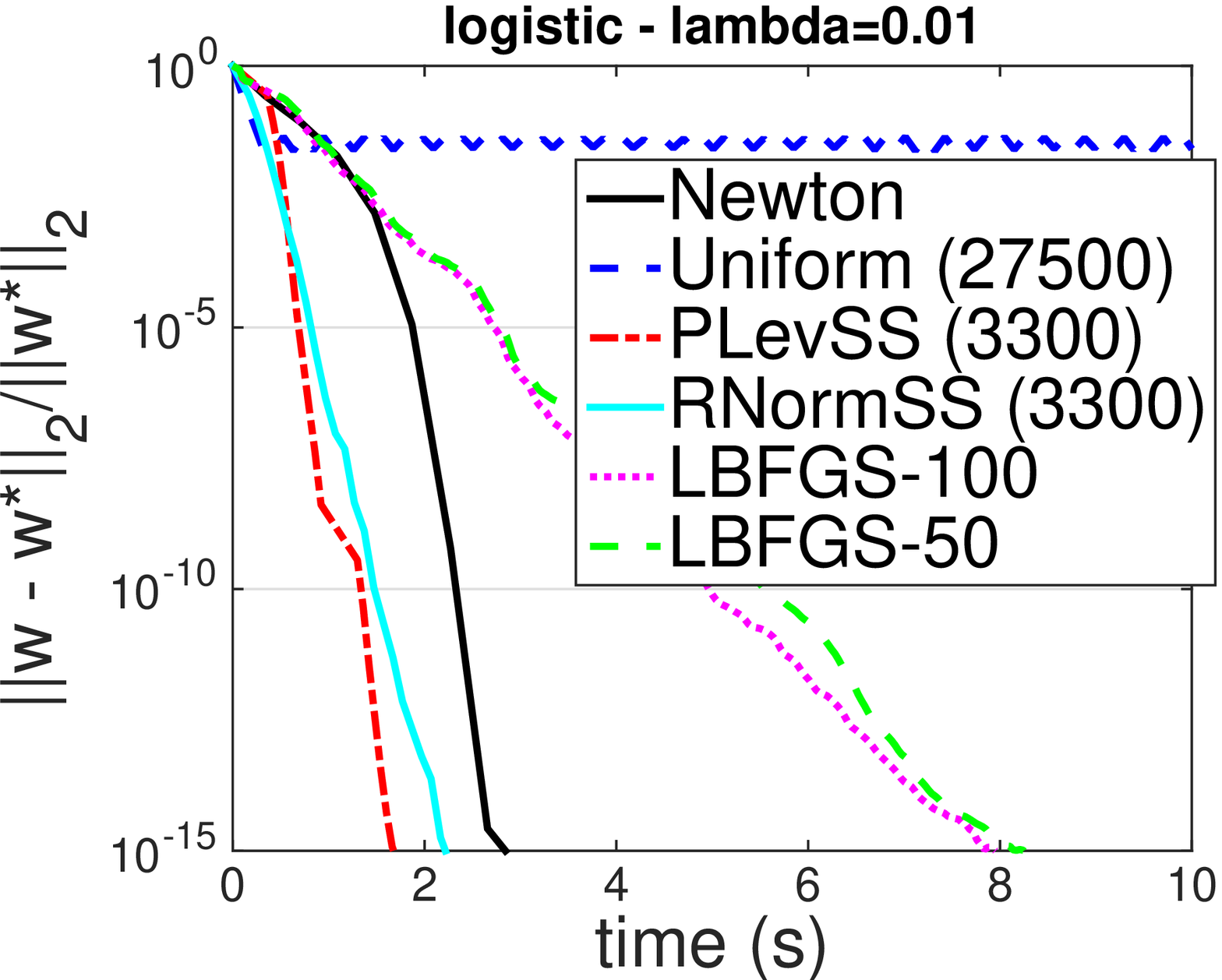}} \\
\subfigure[\texttt{Adult}]{
    \includegraphics[width=0.4\textwidth]{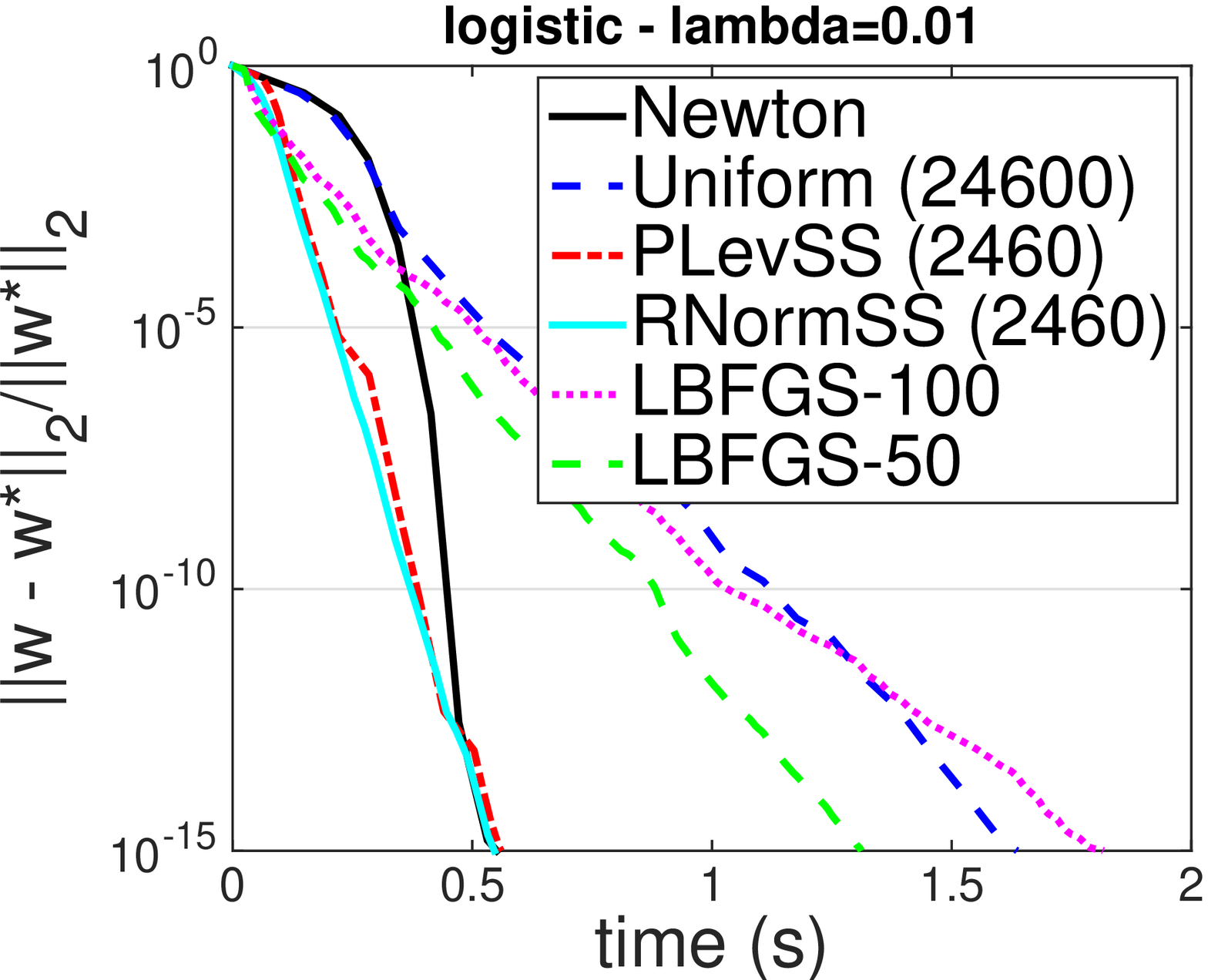}}
\subfigure[\texttt{Buzz}]{
    \includegraphics[width=0.4\textwidth]{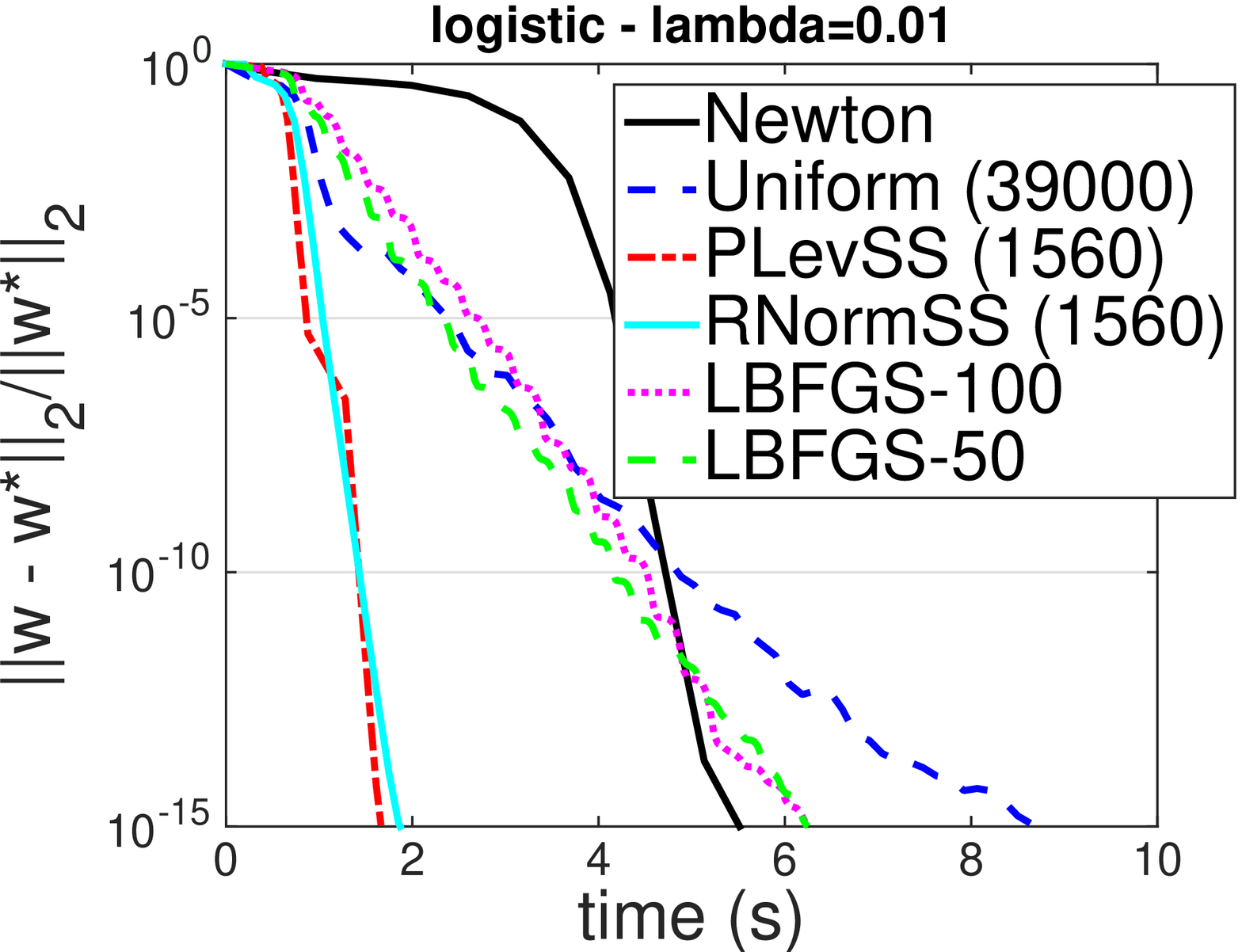}}
\caption{Iterate relative solution error vs.\ time(s) for various second-order methods on four datasets with ridge penalty parameter $\lambda = 0.01$. The values in brackets denote the sample size used for each method.}
\label{fig: logistic}
\end{figure}

\section{Conclusions}\label{sec: conclusions}
In this paper, we propose non-uniformly sub-sampled Newton methods with inexact update for a class of constrained problems. We show that our algorithms have a better dependence on the condition number and enjoy a lower per-iteration complexity, compared to other similar existing methods. Theoretical advantages are numerically demonstrated.

\vspace{5mm}
\textbf{Acknowledgments.}
We would like to acknowledge the Army Research Office and the Defense Advanced Research Projects Agency for providing partial support for this work.


\printbibliography


\newpage
\appendix

\section{Results of Block Partial Leverage Scores}
\label{sec:bplev_results}

In this work, we propose to use a new notion of leverage scores, namely, block partial leverage scores, for approximating matrix of the form $\A^T \A + \Q$. In this section, we give its theoretical guarantee, i.e., quality of approximation, which will be used in the proofs later.

\begin{theorem}
\label{thm:bplev}
Given $\A \in \reals^{N\times d}$ with $n$ blocks, $\Q \in \reals^{d\times d}$ satisfying $\Q \succeq \mathbf{0}$ and $\epsilon \in (0,1)$, let $\{ \tau^\Q_i(\A) \}_{i=1}^n$ be its block partial leverage scores. Let $p_i = \min\{1, s \cdot \frac{\tilde \tau_i}{\sum_{j=1}^n \tilde \tau_j}\}$ with $\tilde \tau_i \geq \tau^\Q_i(\A)$. Construct $\S\A$ by sampling the $i$-th block of $\A$ with probability $p_i$ and rescaling it by $1/\sqrt{p_i}$. Then if 
\begin{equation}
 \label{eq:pblev_s}
  s \geq 2 \left( \sum_{i=1}^n \tilde \tau_i \right) \cdot \left(\|\U^T\D\U\| + \frac{\epsilon}{3}\right) \cdot \log\left(\frac{4d}{\delta}\right) \cdot \frac{1}{\epsilon^2},
\end{equation}
where $\D$ is a diagonal matrix with $\D_{ii} = \tau_i^\Q(\A) / \tilde \tau_i$ and $\U$ is a matrix satisfying $\U^T \U \preceq \mathbf{I}$,
with probability at least $1-\delta$, we have
\begin{equation}
 \label{eq:pblev_obj}
   -\epsilon(\A^T \A + \Q) \preceq \A^T \S^T \S \A - \A^T \A \preceq \epsilon (\A^T \A + \Q).
\end{equation}
\end{theorem}

\begin{proof}
Denote $\bar \A = \begin{pmatrix} \A \\ \Q^{\frac{1}{2}} \end{pmatrix}$.
Let $\bar \A = \bar \U\R$ where $\bar \U$ has orthonormal columns.
Then define $\U = \A \R^{-1}$ and $\U_i = \A_i \R^{-1}$ for $i = 1,\ldots,n$.
By definition, the true partial leverage scores $\tau_i^\Q(\A)$'s are defined as $\tau_i = \tr(\U_i \U_i^T)$. For simplicity, we use $\tau_i$ to denote $\tau_i^\Q(\A)$.

In the following we bound $\|\U^T \S^T \S \U - \U^T \U\| \leq \epsilon$ with high probability.
For $i = 1,\ldots,n$, define
\begin{equation}
  \X_j = \begin{cases}
     (\frac{1}{p_i} - 1) \U_i^T \U_i & \text{with probability } p_i;\\
     - \U_i^T \U_i & \text{with probability } 1 - p_i.
 \end{cases}
\end{equation}
Also define $\Y = \sum_{i=1}^n \X_i$.
We have $\Expect{\X_i} = 0$.
In the following we bound $\| \Y \|$ using matrix Bernstein bound.

First, we bound $\|\X_i\|$ for $i=1,\ldots,n$. 
Let $\mathcal{I} = \{i | p_i < 1 \}$.
If $i \in \mathcal{I}^c$, then $\X_i = 0$ and $\|\X_i\| = 0$. Thus we only consider the case where $i \in \mathcal{I}$. In this case $p_i = \frac{s \tilde \tau_i}{\sum_j \tilde \tau_j}$.
Since $p_i < 1$, we have
\begin{equation}
   -\U_i^T \U_i/p_i \preceq \X_i \preceq \U_i^T \U_i/p_i.
\end{equation}
Moreover,
\begin{eqnarray}
\label{eq:lemma_partial_lev_1}
   \|\X_i\| &\leq& \|\U_i^T \U_i\|/p_i = \|\U_i\|^2 / p_i 
     \\
     &\leq& \frac{\left( \sum_{j=1}^n \tilde \tau_j \right) \|\U_i\|^2}{s \tilde \tau_i}
     = \frac{\left( \sum_{j=1}^n \tilde \tau_j \right) \|\U_i\|^2}{s \tau_i} \frac{\tau_i}{\tilde \tau_i} \\ 
     \label{eq:lemma_partial_lev_1}
     &\leq& \frac{\left( \sum_{j=1}^n \tilde \tau_j \right)}{s} \frac{\tau_i}{\tilde \tau_i} \leq \frac{\left( \sum_{j=1}^n \tilde \tau_j \right)}{s}.
\end{eqnarray}
The last inequality is coming from the fact that $ \tilde \tau_i \geq\tau_i = \tr(\U_i \U_i^T) \geq \|\U_i\|^2$.

Next, we bound $\Expect{\Y^2} = \sum_{i=1}^n \Expect{\X_i^2}$.
We have
  \begin{eqnarray}
    \sum_{i=1}^n \Expect{\X_i^2} &=& \sum_{i\in\mathcal{I}} \Expect{\X_i^2} = \sum_{i=1}^n (p_i\cdot (\frac{1}{p_i}-1)^2 + (1-p_i) ) \A_i^T \A_i \A_i^T \A_i \\
     &=& \sum_{i\in\mathcal{I}} (\frac{1}{p_i} - 1) \U_i^T \U_i \U_i^T \U_i 
     \preceq \sum_{i\in\mathcal{I}} \frac{1}{p_i} \U_i^T \U_i \U_i^T U_i \\
     \label{eq:lemma_partial_lev_2}
     &\preceq& \sum_{i\in\mathcal{I}} \frac{\sum_{j=1}^n \tilde \tau_j}{s} \cdot  \frac{\tau_i}{\tilde \tau_i} \cdot \U_i^T  \U_i \\
     &\preceq& \frac{\sum_{j=1}^n \tilde \tau_j}{s} \cdot \U^T \D \U,
  \end{eqnarray}
where $\D$ is a diagonal matrix with $\D_{ii} = \tau_i / \tilde \tau_i$. Above, \eqref{eq:lemma_partial_lev_2} holds since by \eqref{eq:lemma_partial_lev_1} we have 
$ \U_i^T \U_i/p_i \preceq \left( \sum_{j=1}^n \tilde \tau_j \right) \cdot \frac{\tau_i}{\tilde \tau_i} \frac{1}{s} \cdot \mathbf{I}$.
Therefore,
\begin{equation}
  \| \Expect{\Y^2} \| \leq \frac{\sum_{j=1}^n \tilde \tau_j}{s} \cdot \| \U^T \D \U \|.
\end{equation}
Since $\U$ consists of a subset of rows of $\bar \U$, one can show that $\U^T \U \preceq \bar \U^T \bar \U = \mathbf{I}$.

Given these, by the matrix Bernstein bound\cite{tropp2015introduction}, we have when
\begin{equation}
  s \geq 2 \left( \sum_{j=1}^n \tilde \tau_j \right) \left(\|\U^T\D\U\| + \frac{\epsilon}{3}\right) \cdot \log\left(\frac{4d}{\delta}\right) \cdot \frac{1}{\epsilon^2},
\end{equation}
with probability at least $1-\delta$, 
\begin{equation}
  \| \Y \| \leq \epsilon
\end{equation}
holds. 
With this one can show that
\begin{equation}
  -\epsilon \mathbf{I} \preceq \U^T \S^T \S \U - \U^T \U \preceq \epsilon \mathbf{I}.
\end{equation}
Furthermore, we have
\begin{equation}
  -\epsilon \R^T \R \preceq \R \U^T \S^T \S \U \R- \R^T\U^T \U \R \preceq \epsilon \R^T \R.
\end{equation}
Therefore,
\begin{equation}
  -\epsilon \bar \A^T \bar \A \preceq \A^T \S^T \S \A - \A^T \A \preceq \epsilon \bar \A^T \bar \A.
\end{equation}

This completes the proof.
\end{proof}

\noindent
{\bf Remark.}
In \eqref{eq:pblev_s}, since each element of $\D$ is no greater than $1$ and $\U^T \U \preceq \mathbf{I}$, we have
\begin{equation}
  \U^T \D \U \preceq \U^T \U \preceq \mathbf{I}.
\end{equation}
So the sampling size 
\begin{equation}
  s = 4 \left( \sum_{i=1}^n \tilde \tau_i \right) \cdot \log\left(\frac{4d}{\delta}\right) \cdot \frac{1}{\epsilon^2}
\end{equation}
is sufficient to yield \eqref{eq:pblev_obj}.

\section{Proofs}
\subsection{Proof of Lemma~\ref{lem:general}}
In the section, we will prove the results in Lemma~\ref{lem:general}. Specifically we have two parts of proof. {\bf Part 1} is to prove the case when the subproblem is solved exactly and {\bf Part 2} is for the case when the subproblem is solved approximately. And in {\bf Part 1}, we also have two cases, one is recursion inequality under condition~\eqref{eq:cond1} and the other is the recursion inequality under condition~\eqref{eq:cond2}.

Throughout the proof, we use $\preceq$ to denote the partial order defined on the cone $\{ \mathbf{B} \in \reals^{d\times d} \: | \: \w^T \mathbf{B} \w \geq 0 \:\: \forall \w \in \K\}$.
First, denote $\Delta_t := \w_t - \w^*$, then based on Assumptions~\ref{assump:1} and \ref{assump:2}, we have
\begin{equation}
\|\nabla^2 F(\w_t) - \nabla^2 F(\w^*) \|_2 \le L\|\Delta_t\|_2, \quad \mu \mathbf{I} \preceq \H^* \preceq \nu \mathbf{I}.
\end{equation}
Therefore we can get
\begin{equation}
\label{eq:dt}
(\mu - L\|\Delta_t\|_2)\mathbf{I} \preceq \nabla^2 F(\w_t) \preceq (\nu + L\|\Delta_t\|_2)\mathbf{I}.
\end{equation}
Since $\|\Delta_t\|_2 \le \frac{\mu}{4L}$, then
\begin{equation}
\label{eq:dt}
\frac{3}{4}\mu\mathbf{I} \preceq \nabla^2 F(\w_t) \preceq \frac{5}{4}\nu\mathbf{I}.
\end{equation}
Let
\begin{equation}
\label{eq: sub}
\Psi_t(\w) := \frac{1}{2} (\w - \w_t)^T \tilde \H_t (\w - \w_t) + (\w - \w_t)^T \nabla F(\w_t).
\end{equation}

\paragraph{Part 1} In the part, we consider the subproblem is solved exactly every iteration in Algorithm \ref{alg:main}. Specifically we show the following two results:
\begin{enumerate}[a)]
\item Under condition~\eqref{eq:cond1}, the error recursion~\eqref{eq:rate0} holds with factors in~\eqref{eq:rate1}.
\item Under condition~\eqref{eq:cond2}, the error recursion~\eqref{eq:rate0} holds with factors in~\eqref{eq:rate2}.
\end{enumerate}
If the subproblem $\min_{\w\in\C} \Psi_t(\w)$ is solved exactly in Algorithm \ref{alg:main}, namely
\begin{equation}
\w_{t+1} = \arg\min_{\w \in \C} \Psi_t(\w),
\end{equation}
Then $\Psi_t(\w_{t+1}) \le \Psi_t(\w^*)$. 
By expanding both sides, we have
\begin{equation}
\label{eq:main}
\frac{1}{2} \Delta_{t+1}^T\tilde \H_t\Delta_{t+1} \le \Delta_t^T\tilde \H_t \Delta_{t+1} - \nabla F(\w_t)^T\Delta_{t+1} + \nabla F(\w^*)^T\Delta_{t+1}.
\end{equation}
The third term on the right hand side $\nabla F(\w^*)^T\Delta_{t+1} \ge 0 $ because of the optimality condition.

\begin{align}
RHS &= \Delta_t^T\tilde \H_t \Delta_{t+1} - \nabla F(\w_t)^T\Delta_{t+1} + \nabla F(\w^*)^T\Delta_{t+1}
\\ &= \underbrace{\Delta_t^T\nabla^2 F(\w_t) \Delta_{t+1} - (\nabla F(\w_t) - \nabla F(\w^\ast))^T\Delta_{t+1}}_{T_1} +  \Delta_t^T(\tilde \H_t - \nabla^2 F(\w_t) ) \Delta_{t+1}.
\end{align}

\begin{align}
T_1 &=  \Delta_t^T \nabla^2 F(\w_t) \Delta_{t+1}  - \Delta_t^T \left(\int_0^1 \nabla^2 F(\w^\ast + \delta (\w_t  - \w^\ast) ) d\delta \right) \Delta_{t+1}  
\\
& \le \Delta_t^T \left( \int_0^1\left( \nabla^2 F(\w^\ast + \delta (\w_t  - \w^\ast) )  - \nabla^2 F(\w_t) \right )d\delta \right) \Delta_{t+1}
\\
& \le \|\Delta_t\|_2\cdot \int_0^1 L\delta\|\Delta_t \|_2 d\delta \cdot \|\Delta_{t+1}\|
\\
&\le \frac{L}{2} \cdot \|\Delta_t\|^2 \cdot \|\Delta_{t+1}\|.
\end{align}
Plug in $T_1$ and rewrite \eqref{eq:main}, we get
\begin{align}
\label{eq:lr}
1 \le \frac{L\cdot\|\Delta_t\|^2 \cdot \|\Delta_{t+1}\|}{ \Delta_{t+1}^T\tilde \H_t\Delta_{t+1}} + 2\frac{\Delta_t^T(\tilde \H_t - \nabla^2 F(\w_t) ) \Delta_{t+1}}{ \Delta_{t+1}^T\tilde \H_t\Delta_{t+1}}.
\end{align}
\begin{enumerate}[a)]
\item
Under condition~\eqref{eq:cond1}, we have 
\begin{align}
\lmin^\K(\tilde{\H_t}) &\ge \lmin^\K(\nabla^2 F(\w_t)) - \epsilon\cdot \lmax^\K(\nabla^2 F(\w_t)) 
\\&= (1 - \epsilon\kappa(\nabla^2 F(\w_t)))\cdot \lmin^\K(\nabla^2 F(\w_t))
\\&\ge ( 1- 2\epsilon\kappa)\cdot\lmin^\K(\nabla^2 F(\w_t)).
\end{align}
Then
\begin{align}
\Delta_t^T(\tilde \H_t - \nabla^2 F(\w_t) ) \Delta_{t+1} \ge  [\lmin^\K(\nabla^2 F(\w_t)) - \epsilon\cdot \lmax^\K(\nabla^2 F(\w_t)) ] \cdot \|\Delta_{t+1}\|^2.
\end{align}
Therefore
\begin{eqnarray}
1 &\le& \frac{L\cdot\|\Delta_t\|^2 \cdot \|\Delta_{t+1}\|}{[\lmin^\K(\nabla^2 F(\w_t)) - \epsilon\cdot \lmax^\K(\nabla^2 F(\w_t)) ] \cdot \|\Delta_{t+1}\|^2}\\ &&+ 2 \frac{\epsilon \cdot \lmax^\K(\nabla^2 F(\w_t))\cdot \|\Delta_t\| \cdot \|\Delta_{t+1}\|}{[\lmin^\K(\nabla^2 F(\w_t)) - \epsilon\cdot \lmax^\K(\nabla^2 F(\w_t)) ] \cdot \|\Delta_{t+1}\|^2}
\end{eqnarray}
Reorganize it and combine \eqref{eq:dt} we get
\begin{align}
\|\Delta_{t+1}\| &\le \frac{L}{3\mu/4 - 5\epsilon \nu/4}\cdot \|\Delta_t\|^2 + \frac{5\epsilon\nu/2}{3\mu/4 - 5\epsilon\nu/4}\cdot \|\Delta_t\|
\\&\le \frac{2L}{(1 -2\kappa \epsilon)\mu}\cdot \|\Delta_t\|^2 + \frac{4\kappa\epsilon}{1 - 2\kappa\epsilon}\cdot\|\Delta_t\|.
\end{align}
This proves a).
\item
Under condition \eqref{eq:cond2}, we have
\begin{align}
\Delta_{t+1}^T\tilde\H_t \Delta_{t+1} \ge\Delta_{t+1}^T\nabla^2 F(\w_t) \Delta_{t+1}  - \epsilon\cdot \Delta_{t+1}^T\nabla^2 F(\w_t) \Delta_{t+1} = ( 1 - \epsilon) \cdot \Delta_{t+1}^T\nabla^2 F(\w_t) \Delta_{t+1}.
\end{align}
and 
\begin{align}
\Delta_t^T(\tilde\H_t - \nabla^2 F(\w_t))\Delta_{t+1} \le\epsilon\cdot \sqrt{\Delta_t^T\nabla^2 F(\w_t)\Delta_t}\cdot\sqrt{\Delta_{t+1}^T\nabla^2 F(\w_t)\Delta_{t+1}}
\end{align}
Then plug in \eqref{eq:lr}, we have
\begin{align}
1 &\le \frac{L \cdot\|\Delta_t\|^2 \cdot \|\Delta_{t+1}\|}{( 1 - \epsilon) \cdot \Delta_{t+1}^T\nabla^2 F(\w_t) \Delta_{t+1}}  + 2 \frac{\epsilon\cdot \sqrt{\Delta_t^T\nabla^2 F(\w_t)\Delta_t}\cdot\sqrt{\Delta_{t+1}^T\nabla^2 F(\w_t)\Delta_{t+1}}}{( 1 - \epsilon) \cdot \Delta_{t+1}^T\nabla^2 F(\w_t) \Delta_{t+1}}
\\&\le \frac{L \cdot\|\Delta_t\|^2 \cdot \|\Delta_{t+1}\|}{( 1 - \epsilon) \cdot \lmin^\K(\nabla^2 F(\w_t))\cdot \|\Delta_{t+1}\|^2} + \frac{2\epsilon\cdot \sqrt{\Delta_t^T\nabla^2 F(\w_t)\Delta_t}}{(1-\epsilon)\cdot\sqrt{\Delta_{t+1}^T\nabla^2 F(\w_t)\Delta_{t+1}}}
\\& \le \frac{L}{(1-\epsilon)\cdot \lmin^\K(\nabla^2 F(\w_t))}\cdot\frac{\|\Delta_t\|^2}{\|\Delta_{t+1}\|} + \frac{2\epsilon}{1 - \epsilon}\cdot\sqrt{\frac{\lmax^\K(\nabla^2 F(\w_t))\cdot \|\Delta_t\|^2}{\lmin^\K(\nabla^2 F(\w_t))\cdot\|\Delta_{t+1}\|^2}}
\end{align}
Reorganize it and combine \eqref{eq:dt} we get
\begin{align}
\|\Delta_{t+1}\| &\le \frac{L}{(1-\epsilon)\cdot 3\mu/4}\cdot\|\Delta_t\|^2 + \frac{2\epsilon}{1 - \epsilon}\cdot\sqrt{\frac{5\nu/4}{3\mu/4}}\cdot \|\Delta_t\|
\\ &\le \frac{2L}{(1-\epsilon)\mu} \cdot \|\Delta_t\|^2 + \frac{3\epsilon}{1-\epsilon}\cdot \sqrt{\kappa} \cdot \|\Delta_t\|,
\end{align}
which proves b). 
\end{enumerate}

\paragraph{Part 2} Now we move on to prove the case when the subproblem is approximately solved. Specifically we want to show under condition~\eqref{eq:err_req}, the error recursion \eqref{eq:rate0_inexact} holds where $C_l, C_q$ are the same constants in the case when the problem is solved exactly.

Consider at the iteration $t$ in Algorthm~\ref{alg:main}. First, $\w_{t+1}^* = \argmin_{\w\in\C}\Psi_t(\w)$ (Note that here $\w_{t+1}$ is not the minimizer any more). Then based on the proof in {\bf Part 1}, we have
\begin{align}
\|\w_{t+1}^* - \w^*\| \le C_q\cdot\|\w_t - \w^*\|^2 + C_l \cdot \|\w_t - \w^*\|.
\end{align}
Therefore
\begin{align}
\|\w_{t+1} - \w^*\| &\le \|\w_{t+1} - \w^*_{t+1}\| + \|\w_{t+1}^* - \w^*\|
\\ & \le \epsilon_0\cdot\|\w^*_{t+1} - \w_t \| + \|\w_{t+1}^* - \w^*\| 
\\ & \le \epsilon_0\cdot\|\w^*_{t+1} - \w^*\| + \epsilon_0\cdot \|\w_t - \w^*\| + \|\w_{t+1}^* - \w^*\|
\\ & = (1+ \epsilon_0)\cdot \|\w^*_{t+1} - \w^*\| + \epsilon_0 \cdot \|\w_t - \w^*\| 
\\ & \le (1+\epsilon_0)C_q \cdot \| \w_t - \w^*\|^2 + (\epsilon_0 + (1+\epsilon_0)C_l) \cdot \|\w_t - \w^*\|.
\end{align}
And this completes the proof.

\subsection{Proof of Theorem~\ref{lem:lev_const}}
Due to Assumption~\ref{assump:pos}, it takes $\bigO(\nnz(\A))$ time to construct the augmented matrix $\A \in \reals^{nk \times d}$. Here we apply a variant of the algorithm in~\citep{drineas2012fast} by using the sparse subspace embedding~\cite{clarkson13sparse} as the underlying sketching method. One can show that, with high probability, it takes $\bigO(\nnz(\A))d\log(nk+d)) = \bigO(\nnz(\A)\log n)$ time to compute a set of approximate leverage scores with constant approximation factor.
This completes the proof.

\subsection{Proof of Theorem~\ref{thm:lev}}
Based on Lemma~\ref{lem:equv_cond} (stated below), we convert condition~\eqref{eq:cond2} to a standard matrix product approximation guarantee.
A direct corollary of Theorem~\ref{thm:bplev} completes the proof.


\begin{lemma}
\label{lem:equv_cond}
Given $\A \in \reals^{N\times d}$ with $n$ blocks, $\Q \in \reals^{d\times d}$ satisfying $\Q \succeq \mathbf{0}$ and $\epsilon \in (0,1)$, and a sketching matrix $\S\in\reals^{s\times n}$, the following two conditions are equivalent:
\begin{enumerate}[(a)]
\item 
\begin{equation}
 \label{eq:pblev_obj1}
   -\epsilon(\A^T \A + \Q) \preceq \A^T \S^T \S \A - \A^T \A \preceq \epsilon (\A^T \A + \Q).
\end{equation}
\item
\begin{equation}
\label{eq:pblev_obj2}
|\x^T(\A^T\S^T\S\A - \A^T\A)\y | \le \epsilon \cdot \sqrt{\|\A\x\|_2^2 + \x^T\Q\x}\cdot \sqrt{ \|\A\y\|_2^2 + \y^T\Q\y}, ~~\forall \x,\y \in \reals^d.
\end{equation}
\end{enumerate}
\end{lemma}

\begin{proof}
First, it is straightforward to see (b) $\Rightarrow$ (a) by setting $\x = \y$ in \eqref{eq:pblev_obj2}. So now we prove the other direction. 

Denote $\bar \A = \begin{pmatrix} \A \\ \Q^{\frac{1}{2}} \end{pmatrix}$.
Let $\bar \A = \bar \U\R$ where $\bar \U$ has orthonormal columns.
Then define $\U = \A \R^{-1}$ and $\U_i = \A_i \R^{-1}$ for $i = 1,\ldots,n$.
Then \eqref{eq:pblev_obj1} is equivalent to 
\begin{equation}
-\epsilon \bar \A^T \bar \A \preceq \A^T \S^T \S \A - \A^T \A \preceq \epsilon \bar \A^T \bar \A.
\end{equation}
Since $\R^{-1}$ is a full rank matrix, then
\begin{equation}
-\epsilon \R^{-T}\bar \A^T \bar \A \R^{-1}\preceq \R^{-T}(\A^T \S^T \S \A - \A^T \A)\R^{-1} \preceq \epsilon \R^{-T}\bar \A^T \bar \A\R^{-1}.
\end{equation}
i.e.
\begin{equation}
-\epsilon \mathbf{I}  = -\epsilon \bar\U^T \bar \U \preceq \U^T \S^T \S \U - \U^T \U\preceq \epsilon \bar \U^T \bar \U = \epsilon \mathbf{I}.\quad\text{[$\bar\U$ is orthonormal bases] }
\end{equation}
Therefore 
\begin{align}
\| \U^T \S^T \S \U - \U^T \U\| \le \epsilon.
\end{align}
Now consider
\begin{align}
|\x^T(\A^T\S^T\S\A - \A^T\A) \y|  &= |\x^T\R^T(\U^T\S^T\S \U - \U^T\U)\R\y|
\\
&\le \|\R\x\|_2 \cdot \|\U^T\S^T\S \U - \U^T\U\|_2\cdot\|\R\y\|_2
\\
& \le \epsilon \cdot \|\bar\U\R\x\|_2 \cdot \|\bar\U\R\y\|_2  \quad\text{[$\bar\U$ is orthonormal bases] }
\\
& = \epsilon \cdot \|\bar\A\x\|_2 \cdot \|\bar\A\y\|_2
\\
& = \epsilon \cdot \sqrt{\|\A\x\|_2^2 + \x^T\Q\x} \cdot \sqrt{\|\A\y\|_2^2 + \y^T\Q\y}.
\end{align}
This completes the proof.
\end{proof}

\subsection{Proof of Corollary~\ref{cor:complexity}}
According to Lemma~\ref{lem:general}, we have the follow error recursion
\begin{align}
\|\w_{t+1} - \w^*\| &\le (1+\epsilon_0)C_q \cdot \|\w_t - \w^*\|^2 + (\epsilon_0 + (1+\epsilon_0)C_l)\cdot \|\w_t - \w^*\| 
\\ &\le \left [(1+\epsilon_0) C_q \frac{\mu}{4L} + \epsilon_0 + (1+\epsilon_0)C_l\right]\cdot \|\w_t -\w^*\|.
\end{align}
If leverage scores sampling is used, then
\begin{align}
C_q = \frac{2L}{(1-\epsilon)\mu}, ~~C_l = \frac{3\epsilon\sqrt{\kappa}}{1-\epsilon}.
\end{align}
Therefore
\begin{align}
\|\w_{t+1} - \w^*\| &\le \left [(1+\epsilon_0) C_q \frac{\mu}{4L} + \epsilon_0 + (1+\epsilon_0)C_l\right]\cdot \|\w_t -\w^*\|
\\ &= \left[\frac{1 + \epsilon_0}{2(1-\epsilon)}+ \epsilon_0 + (1 + \epsilon_0)  \frac{3\epsilon\sqrt{\kappa}}{1-\epsilon} \right] \cdot \|\w_t - \w^*\|
\\ & = \rho \cdot \| \w_t - \w^*\|.
\end{align}
Now choose $\epsilon \le \min\left\{\frac{1}{10\sqrt{\kappa}},0.1\right\}$ and $\epsilon_0 \le 0.01$, then we can get $\rho < 0.9$.

Since we use CG, in order to achieve $\epsilon_0 \le 0.01$, we need $\tilde\bigO(sd\sqrt{\kappa})$ due to Table~\ref{tab:solver} (Note that since we are in the local region, $\tilde{\kappa_t} = \Theta(\kappa)$).
Meanwhile, since $\epsilon \le\min\left\{\frac{1}{10\sqrt{\kappa}},0.1\right\}$, then $s = \tilde\bigO(d/\epsilon^2) = \tilde\bigO(d\kappa)$.

Therefore, the complexity per iteration is 
\begin{align}
t_{const} + t_{solve} = \tilde\bigO(\nnz(\A)) + \tilde\bigO(d^2\kappa^{3/2}) = \tilde\bigO(\nnz(\A) +d^2\kappa^{3/2}) .
\end{align}

Similarly, if block norm squares sampling is used, then
\begin{align}
C_q = \frac{2L}{(1-2\epsilon\kappa)\mu}, ~~C_l = \frac{4\epsilon\kappa}{1-2\epsilon\kappa}.
\end{align}
Therefore
\begin{align}
\|\w_{t+1} - \w^*\| &\le \left [(1+\epsilon_0) C_q \frac{\mu}{4L} + \epsilon_0 + (1+\epsilon_0)C_l\right]\cdot \|\w_t -\w^*\|
\\ &= \left[\frac{1 + \epsilon_0}{2(1-2\epsilon\kappa)}+ \epsilon_0 + (1 + \epsilon_0)  \frac{4\epsilon\kappa}{1-2\epsilon\kappa} \right] \cdot \|\w_t - \w^*\|
\\ & = \rho \cdot \| \w_t - \w^*\|.
\end{align}
Now choose $\epsilon \le\min \left\{\frac{1}{10\kappa},0.1\right\}$ and $\epsilon_0 \le 0.01$, then we can get $\rho < 0.9$.
Similar to the case using leverage scores sampling, we get the total complexity per iteration which is
\begin{align}
t_{const} + t_{solve} = \tilde\bigO(\nnz(\A))  + \tilde\bigO({\bf sr}(\A)d \kappa^{5/2}) = \tilde\bigO(\nnz(\A) + {\bf sr}(\A)d \kappa^{5/2}).
\end{align}


\subsection{Proof of Theorem~\ref{thm:unif}}
For $j = 1,\ldots,s$, define
\begin{equation}
  \X_j =
     \frac{1}{s} (n \A_i^T \A_i - \A^T\A) ~~  \text{with probability }1/n, \forall i.
\end{equation}
And $\Y = \sum_{j=1}^s \X_j (= \A^T\S^T\S\A - \A^T\A$). Then $\Expect{\X_j} = 0$. In the following we will bound $\|\Y\|$ through matrix Bernstein inequality. 
For convenience, let's denote $K_t := \max_i\|\A_i\|^2$.

First,
\begin{align}
\|\X_j\|  = \frac{1}{s}\|n\A_i^T\A_i - \A^T\A\| \le 2nK_t/s.
\end{align}
\begin{align}
\Expect {\X_j^2} &= \frac{1}{ns^2} \sum_{i=1}^n(n \A_i^T \A_i - \A^T\A)^2 
\\ & = \frac{1}{ns^2} \sum_{i=1}^n \left [n^2 (\A_i^T\A_i)^2 + (\A^T\A)^2 - 2n\A_i^T\A_i\A^T\A  \right ]
\\ & = \frac{1}{s^2} (\A^T\A)^2 + \frac{1}{s^2} \sum_{i=1}^n n(\A_i^T\A_i)^2  - \frac{2}{s^2}(\A^T\A)^2
\\ & \preceq \frac{1}{s^2} \sum_{i=1}^n n(\A_i^T\A_i)^2
\\ & \preceq \frac{1}{s^2} nK_t \A^T\A. 
\end{align}
Therefore,
\begin{align}
\| \Expect{\Y^2}\| \le \| \frac{1}{s} nK_t \A^T\A\| = \frac{1}{s}nK_t \|\A\|^2.
\end{align}
By the matrix Bernstein bound~\cite{tropp2015introduction},  we have when
\begin{align}
s \ge 4nK_t ( \|\A\|^2 + \frac{\epsilon_0}{3}) \cdot \log(\frac{d}{\delta}) \cdot \frac{1}{\epsilon_0^2},
\end{align}
with probability at least $1 - \delta$, 
\begin{align}
\|Y\| \le \epsilon.
\end{align}

Now choose scale the $\epsilon_0 = \|A\|^2 \epsilon$, where $\epsilon \in (0,1)$, then 
\begin{align}
s \ge 4\frac{nK_t}{\|\A\|^2}  \cdot \log(\frac{d}{\delta}) \cdot \frac{1}{\epsilon^2} = 4n\frac{\max_i\|\A_i\|^2}{\|\A\|^2}  \cdot \log(\frac{d}{\delta}) \cdot \frac{1}{\epsilon^2},
\end{align}
with probability at least $1-\delta$, 
$ \|\A^T\S^T\S\A - \A^T\A\| \le \epsilon\cdot \|\A^T\A\| $ holds.
Since $\Q \succeq \mathbf{0}$, then $\|\A^T\A \| \le \| \A^T\A + \Q\|$. Therefore condition \eqref{eq:cond1} holds. And this completes the proof.

\end{document}